\documentclass[12pt]{article}

\usepackage{geometry,float}
\usepackage{amssymb}
\usepackage{amstext}
\usepackage{bm}
\usepackage{graphics}
\usepackage{graphicx}
\usepackage{amsthm}
\usepackage{amsmath}
\usepackage{latexsym}
\usepackage{rotate}
\usepackage{lscape}
\usepackage{color}
\usepackage{epsfig,epsf,psfrag}
\usepackage{sectsty}
\usepackage{booktabs}
\usepackage{graphics}
\usepackage{epsfig}
\usepackage{multirow}
\usepackage{graphicx,xr-hyper,hyperref}
\usepackage{amsmath}
\usepackage{amsthm}
\usepackage{amssymb}
\usepackage{sectsty}
\usepackage{epsfig,natbib}
\usepackage{rotate}
\usepackage{lscape}
\usepackage{setspace}
\usepackage{subcaption}
\usepackage{float}
\usepackage{graphicx}
\usepackage{mathtools}
\usepackage{enumitem}
\usepackage{amssymb, amsmath, amsthm, graphicx, bbm, mathtools}
\usepackage{enumitem}
\setlist{nolistsep}
\usepackage[dvipsnames]{xcolor}
\usepackage{tikz}
\usetikzlibrary{arrows.meta}
\usetikzlibrary{decorations.pathreplacing}
\usepackage{csquotes}

\usepackage{algorithm}
\usepackage{algpseudocode}
\algnewcommand\algorithmicinput{\textbf{INPUT:}}
\algnewcommand\INPUT{\item[\algorithmicinput]}
\algnewcommand\algorithmicoutput{\textbf{OUTPUT:}}
\algnewcommand\OUTPUT{\item[\algorithmicoutput]}

\usepackage{xr-hyper}
\usepackage{hyperref}[]
\hypersetup{
    colorlinks=true,
    linkcolor=blue,
    filecolor=magenta,      
    urlcolor=cyan,
      citecolor=blue,
    }
\usepackage{cleveref}

\def\argmin{\mathop{\rm argmin}}
\def\argmax{\mathop{\rm argmax}}
\theoremstyle{plain}
\newtheorem{theorem}{Theorem}
\newtheorem{proposition}[theorem]{Proposition}
\newtheorem{lemma}[theorem]{Lemma}

\theoremstyle{remark}
\newtheorem{remark}{Remark}
\newtheorem{definition}{Definition}
\newtheorem{assumption}{Assumption}

\newcommand{\bc}{\begin{center}}
\newcommand{\ec}{\end{center}}

\bibpunct{(}{)}{;}{a}{}{,}

\begin{document}


\bc {\sc \Large  Localising change points in \\piecewise polynomials of general degrees} \ec


\bc Yi Yu$^\dagger$, Sabyasachi Chatterjee$^\ddagger$, Haotian Xu$^\dagger$ \ec 

\bc $^\dagger$Department of Statistics, University of Warwick \ec

\bc $^\ddagger$Department of Statistics, University of Illinois at Urbana-Champaign \ec

\centerline{December 2021}

\begin{abstract}
In this paper we are concerned with a sequence of univariate random variables with piecewise polynomial means and independent sub-Gaussian noise.  The underlying polynomials are allowed to be of arbitrary but fixed degrees.  All the other model parameters are allowed to vary depending on the sample size.  
    
We propose a two-step estimation procedure based on the $\ell_0$-penalisation and provide upper bounds on the localisation error.  We complement these results by deriving global information-theoretic lower bounds, which show that our two-step estimators are nearly minimax rate-optimal.  We also show that our estimator enjoys near optimally adaptive performance by attaining individual localisation errors depending on the level of smoothness at individual change points of the underlying signal.   In addition, under a special smoothness constraint, we provide a minimax lower bound on the localisation errors.  This lower bound is independent of the polynomial orders and is sharper than the global minimax lower bound.
\end{abstract}

\section{Introduction}

We are concerned with the data $y = (y_1, \ldots, y_n)^{\top} \in \mathbb{R}^n$.  For each $i \in \{1, \ldots, n\}$,
	\begin{equation}\label{eq-y-intro}
		y_i = f(i/n) + \varepsilon_i, 
	\end{equation}
	where $f:\, [0, 1] \to \mathbb{R}$ is an unknown piecewise-polynomial function and $\varepsilon_i$'s are independent mean zero sub-Gaussian random variables.  To be specific, associated with $f(\cdot)$, there is a sequence of strictly increasing integers $\{\eta_k\}_{k = 0}^{K+1}$, with $\eta_0 = 1$ and $\eta_{K+1} = n+1$, such that $f(\cdot)$ restricted on each interval $[\eta_k/n, \eta_{k+1}/n)$, $k = 0, \ldots, K$, is a polynomial of degree at most $r \in \mathbb{N}$.	The maximum degree $r$ is assumed to be arbitrary but fixed, and the number of change points $K$ is allowed to diverge as the sample size $n$ grows unbounded.  The goal of this paper is to estimate $\{\eta_k\}_{k = 1}^K$, called the change points of $f(\cdot)$, accurately and to understand the fundamental limits in detecting and localising these change points.  More detailed model descriptions can be found in \Cref{sec-main-result}.

The work in this paper falls within the general topic of change point analysis, which has a long history and is being actively studied till date.  In change point analysis, one assumes that the underlying distributions change at a set of unknown time points, called change points, and stay the same between two consecutive change points.  A closely related problem is change point detection in piecewise constant signals.  This is studied thoroughly in \cite{chan2013}, \cite{FrickEtal2014}, \cite{dumbgen2001multiscale}, \cite{dumbgen2008multiscale}, \cite{LiEtal2017}, \cite{jeng2012simultaneous} and \cite{wang2020univariate}, among others.  \cite{fearnhead2019detecting}, \cite{baranowski2016narrowest}, \cite{chen2020jump}, \cite{anastasiou2021detecting}, \cite{fryzlewicz2020narrowest}, \cite{maeng2019detecting}  and \cite{cheng2008kernel}  studied change point analysis in piecewise linear signals.  Our work in this paper can be seen as a generalisation of the aforementioned results, allowing for polynomials of arbitrary degrees, and the magnitudes of coefficients changes to vanish as the sample size grows unbounded, although some of the aforementioned work may contain more general assumptions on the noise structure.  Detailed comparisons with some existing literature will be provided after we present our main results.  

Beyond univariate sequence, the existing work on change point analysis includes studies on high-dimensional models \citep[e.g.][]{wang2017optimal, DetteEtal2018, wang2016high}, network models \citep[e.g.][]{wang2018optimal, CribbenYu2017, bhattacharjee2018change}, nonparametric models \citep[e.g.][]{padilla2019optimal, padilla2019optimal2, garreau2018consistent} and regression models \citep[e.g.][]{bai1998estimating, bai2003computation, wang2019localizing, wang2020detecting, rinaldo2021localizing}.

 To divert slightly, it is worth mentioning that instead of focusing on estimating the locations of the change points, a complementary problem is to estimate the whole of the underlying piecewise polynomial function itself.  This is a canonical problem in nonparametric regression and also has a long history.  The piecewise polynomial function is typically assumed to satisfy certain regularity at the change points.  The classical settings therein assume that the degrees of the underlying polynomials are taken to be some particular values and the change points, referred to as knots, are at fixed locations, see e.g.~\cite{green1993nonparametric} and \cite{wahba1990spline}.  More recent regression methods have focussed on fitting piecewise polynomials where the knots are not fixed beforehand and is estimated from the data \citep[e.g.][]{mammen1997locally, tibshirani2014adaptive, shen2020phase, guntuboyina2020adaptive}.   

In this paper, we focus on estimating the locations of the change points accurately, allowing for general and different degrees of polynomials within $f(\cdot)$, diverging number of change points, and different smoothness at different change points.  This framework, to the best of our knowledge, is the most flexible one in both change point analysis and spline regression analysis areas.  In the rest of this paper, we first formalise the problem and introduce the algorithm in \Cref{sec-problem-setup}, followed by a list of contributions in \Cref{sec-contributions}.  The main results are collected in \Cref{sec-main-result}, with more discussions in \Cref{sec-discussion} and the proofs in the Appendices. Extensive numerical experiments are presented in \Cref{sec-numeric}.

\subsection{The problem setup and the description of the estimator}\label{sec-problem-setup}

In order to estimate the change points of $f(\cdot)$, we propose a two-step estimator.  The estimator is defined in this subsection, following introduction of necessary notation used throughout this paper. 

Let $\Pi$ be any interval partition of $\{1, \ldots, n\}$, i.e.~a collection of $|\Pi| \geq 1$ disjoint subsets of $\{1, \ldots, n\}$,
	\[
		\Pi = \left\{\{1, \ldots, s_1 - 1\}, \, \{s_1, \ldots, s_2 - 1\}, \ldots, \{s_{|\Pi|-1}, \ldots, n\}\right\},
	\]
	for some integers $1 = s_0 < s_1 < \ldots s_{|\pi|-1} \leq n < s_{|\Pi|} = n+1$, with $|\cdot|$ denoting the cardinality of a set.  For any such partition $\Pi$, we denote $\eta(\Pi) = \{s_1, \ldots, s_{|\Pi|-1}\}$ to be its change points.  Let $\mathcal{P}_n$ be the collection of all such interval partitions of $\{1, \ldots, n\}$.    
	
For any fixed $\lambda > 0$ and given data $y \in \mathbb{R}^n$, let the estimated partition be
	\begin{equation}\label{eq:defplse}
		\widehat{\Pi} \in \argmin_{\Pi: \, \Pi \in \mathcal{P}_n} G(\Pi, \lambda),
	\end{equation}	
	where 
	\begin{align}\label{eq-def-G-loss}
		G(\Pi, \lambda) = \sum_{I \in \Pi} \|y_I - P_I y_I\|^2 + \lambda|\Pi| = \sum_{I \in \Pi} H(y, I) + \lambda |\Pi|,
	\end{align}
	the notation therein is introduced below.
	\begin{itemize}
	\item The norm $\|\cdot\|$ denotes the $\ell_2$-norm of a vector.
	\item For any interval $I = \{s, \ldots, e\} \subset \{1, \ldots, n\}$, let $y_I = (y_i, i \in I)^{\top} \in \mathbb{R}^{|I|}$ be the data vector on interval $I$ and $P_I$ be the projection matrix 
	\begin{equation}\label{eq-P-I-definition}
		P_I = U_{I, r}(U_{I, r}^{\top}U_{I, r})^{-1}U_{I, r}^{\top},
	\end{equation}
	with
	\begin{equation}\label{eq-U-i-r}
		U_{I, r} = \left(\begin{array}{cccc}
				1 & s/n & \cdots & (s/n)^r \\
				\vdots & \vdots & \vdots & \vdots \\
				1 & e/n & \cdots & (e/n)^r  
			\end{array}\right) \in \mathbb{R}^{(e-s+1) \times (r+1)}.
	\end{equation}

	\end{itemize}

We can see that the loss function $G(\cdot, \cdot)$ is a penalised residual sum of squares.  The penalisation is imposed on the cardinality of the partition, which is in fact an $\ell_0$ penalisation.  The residual sum of squares are the residuals after projecting data onto the discrete polynomial space. 	 \textbf{The initial estimators} $\{\widetilde{\eta}_k\}_{k = 1}^{\widehat{K}}$ are defined to be $\eta(\widehat{\Pi})$, the change points of $\widehat{\Pi}$.

With the estimated partition $\widehat{\Pi}$ and its associated change points $\eta(\widehat{\Pi})$, provided that $|\eta(\widehat{\Pi})| \geq 1$, we proceed to the second-step estimation.  For any $k \in \{1, \ldots, \widehat{K}\}$, let
	\begin{equation}\label{eq-sk-ek-def}
		s_k = \widetilde{\eta}_{k-1}/2 + \widetilde{\eta}_k/2, \quad e_k = \widetilde{\eta}_k/2 + \widetilde{\eta}_{k+1}/2 \quad \mbox{and}\quad I_k = [s_k, e_k),
	\end{equation}
	with $\widetilde{\eta}_0 = 1$ and $\widetilde{\eta}_{\widehat{K}+1} = n+1$.  For any $k \in \{1, \ldots, \widehat{K}\}$, we define 
	\begin{equation}\label{eq-refined-estimators}
		\widehat{\eta}_k = \argmin_{t \in I_k \setminus \{s_k\}} \left\{H(y, [s_k, t)) + H(y, [t, e_k))\right\},
	\end{equation}
	where $H(\cdot, \cdot)$ is defined in \eqref{eq-def-G-loss}.  The updated estimators $\{\widehat{\eta}_k\}_{k = 1}^{\widehat{K}}$ are our \textbf{final estimators}. 
	
As a summary, this two-step algorithm precedes with the optimisation problem \eqref{eq:defplse}, providing a set of initial estimators $\{\widetilde{\eta}_k\}_{k = 1}^{\widehat{K}}$.  With the initial estimators, a parallelisable second step works on every triplet $(\widetilde{\eta}_{k-1}, \widetilde{\eta}_k, \widetilde{\eta}_{k+1})$, $k \in \{1, \ldots, K\}$, to refine $\widetilde{\eta}_k$ and yield $\widehat{\eta}_k$.  This update does not change the number of estimated change points.  Note that the choice of $1/2$ in the definitions of $s_k$'s and $e_k$'s in \eqref{eq-sk-ek-def} is arbitrary, and any constant $c \in (0, 1)$ would work.
	
To help further referring back to our two-step algorithm, we present the full procedure in \Cref{algorithm:PDP}.

\begin{algorithm}[htbp]
\begin{algorithmic}
	\INPUT Data $\{y_i\}_{i=1}^{n}$, tuning parameters $\lambda > 0$.
	\State $\widehat{\Pi} \leftarrow \argmin_{\Pi:\, \Pi \in \mathcal{P}_n} G(\Pi, \lambda)$ \Comment{See \eqref{eq-def-G-loss}}
	\State $\mathcal{B} \leftarrow \eta(\widehat{\Pi})$ \Comment{The initial estimators}
	\State
	\If{$\mathcal{B} \neq \emptyset$}
		\State $\{\widehat{\eta}_k\}_{k = 1}^{\widehat{K}} \leftarrow$ Update $\mathcal{B}$ based on \eqref{eq-refined-estimators} \Comment{The final estimators}
	\EndIf
\caption{Two-step estimation}
\label{algorithm:PDP}
\end{algorithmic}
\end{algorithm} 

We conclude this subsection with two remarks, on the optimisation problem \eqref{eq:defplse} and the computational aspect of the upper bound on the polynomial degree $r$, respectively.
	
\begin{remark}[The optimisation problem \eqref{eq:defplse}]
The uniqueness of the solution \eqref{eq:defplse} is not guaranteed in general, but the properties we are to present regarding the change point estimators hold for any solutions.  In fact, under some mild conditions, for instance the existence of densities of the noise distribution, one can show the minimiser of \eqref{eq:defplse} is unique almost surely \citep[e.g.~Remark 4 in][]{wang2020univariate}.

The optimisation problem \eqref{eq:defplse}, with a general loss function, is known as the minimal partitioning problem \citep[e.g.~Algorithm 1 in][]{FriedrichEtal2008}, which is related with the Schwarz Information Criterion \citep[e.g.][]{yao1989least},  and can be solved by a dynamic programming approach in polynomial time.  The computational cost is of order $O(n^2\mathrm{Cost}(n))$, where $\mathrm{Cost}(n)$ is the computational cost of calculating $G(\Pi, \lambda)$, for any given $\Pi$ and $\lambda$.  To be specific, for \eqref{eq:defplse}, $\mathrm{Cost}(n) = O(n)$, where the hidden constants depend on the polynomial degree $r$, therefore the total computationa cost is $O(n^3)$.  A reference where the computational cost and the dynamic programming algorithm is explicitly mentioned is Lemma 1.1 in \cite{chatterjee2019adaptive}. 

We would like to mention that the minimal partitioning problem has previously being used in change point analysis literature for other models, including \cite{fearnhead2018changepoint}, \cite{killick2012optimal}, \cite{wang2020univariate}, \cite{wang2019localizing} and \cite{wang2020detecting}, among others.  In the spline regression analysis area, the $\ell_0$ penalisation is also exploited, for instance, \cite{shen2020phase} and \cite{chatterjee2019adaptive}, to name but a few.     We would like to reiterate that \cite{shen2020phase} and \cite{chatterjee2019adaptive} studied the estimation risk of the whole underlying functions.  The results derived in this paper focus on the change point localisation.
\end{remark}

\begin{remark}[The polynomial degree upper bound $r$]
The degree $r$ is in fact an input of the algorithm.  One needs to specify the degree $r$ in \eqref{eq:defplse} and \eqref{eq-refined-estimators}.  Usually, when we define a degree-$d$ polynomial, we let 
	\[
		g(x) = \sum_{l = 0}^d c_l x^l, \quad x \in \mathbb{R},
	\]
	with $\{c_l\}_{l = 0}^d \subset \mathbb{R}$ and $c_d \neq 0$.  If $c_d = 0$, then $g(\cdot)$ is regarded as a degenerate degree-$d$ polynomial.  In this paper, we do not emphasis on the highest degree coefficient being nonzero.  With this flexibility, in practice, as long as the input $r$ is not smaller than the largest degree of the underlying polynomials, then the performances of the algorithm are still guaranteed.  However, the larger the input $r$ is, the more costly the optimisation is.  More regarding this point will be discussed after we present our main theorem.
\end{remark}

\subsection{Main contributions}\label{sec-contributions}

To conclude this section, we summarise our contributions in this paper.

Firstly, to the best of our knowledge, this is the first paper studying the change point localisation in piecewise polynomials with general degrees.  The model we are concerned in this paper enjoys great flexibility.  We allow for the number of change points and the variances of the noise sequence to diverge, and the differences between two consecutive different polynomials to vanish, as the sample size grows unbounded.

Secondly, we propose a two-step estimation procedure for the change points, detailed in \Cref{algorithm:PDP}.  The first step is a version of the minimal partitioning problem \citep[e.g.][]{FriedrichEtal2008}, and the second step is a parallelisable update.  The first step can be done in $O(n^3)$ time and the second step in $O(n)$ time.  

Thirdly, we provide theoretical guarantees for the change point estimators returned by \Cref{algorithm:PDP}.  To the best of our knowledge, it is the first time in the literature, establishing the change point localisation rates for piecewise polynomials with general degrees.  Prior to this paper, the state-of-the-art results were only on piecewise linear signals.  In this paper, we allow the underlying contiguous polynomials to pertain different smoothness at different change points.  This is reflected in our localisation error bound for each individual change point.  In short, we show that our change point estimator enjoys nearly optimal adaptive localisation rates.  In addition to the global minimax rates, we have also derived minimax rates on the localisation errors when restricting to some special classes.  These results again, are the first time shown for the general polynomials.

Lastly, in a fully finite sample framework, we provide information-theoretic lower bounds characterising the fundamental difficulties of the problem, showing that our estimators are nearly minimax rate-optimal.  To the best of our knowledge, even for the piecewise linear case, previous minimax lower bounds only focused on the scaling in the sample size $n$ whereas we derive a minimax lower bound involving all the parameters of the problem.  More detailed comparisons with existing literature are in \Cref{sec-discussion}.

\section{Main Result}\label{sec-main-result}

In this section, we investigate the theoretical properties of the initial and the final estimators returned by \Cref{algorithm:PDP}.  

\subsection{Characterising differences between different polynomials}\label{sec-character}

In the change point analysis literature, the difficulty of the change point estimation task can be characterised by two key model parameters: the minimal spacing between two consecutive change points and the minimal difference between two consecutive underlying distributions.  In this paper, the underlying distributions are determined by the polynomial coefficients.  For two different at-most-degree-$r$ polynomials, the difference is nailed down to the difference between two $(r+1)$-dimensional vectors, consisting of the polynomial coefficients.  To characterise the difference, for any integers $r, K \geq 0$, we adopt the following reparameterising for any piecewise polynomial function $f(\cdot) \in \mathcal{F}^{r, K}_n$, where
	\begin{align}
		& \mathcal{F}_n^{r, K}  = \Big\{f(\cdot): [0, 1] \to \mathbb{R}: 1 = \eta_0 < \eta_1 < \cdots < \eta_K = n < \eta_{K+1} = n+1, \nonumber \\
		& \mbox{s.t. } \forall k \in \{0, 1, \ldots, K\}, \, f_{[\eta_k/n, \eta_{k+1}/n)}: [\eta_k/n, \eta_{k+1}/n) \to \mathbb{R}, \nonumber \\
		& \mbox{with } f|_{[\eta_k/n, \eta_{k+1}/n)}(x) = f(x), \nonumber \\
		& \mbox{is a right-continuous with left limit polynomial of degree at most } r.\Big\}. \label{eq-def-f-r-n-k}
	\end{align}

\begin{remark}[Uniqueness of the change points]
	Note that if two adjacent different polynomials are continuous at the change point, then the definition of change point is not necessarily unique and may differ by the degree of the polynomials.  In this case, either choice satisfying \eqref{eq-def-f-r-n-k} can serve as a change point and will not affect the theoretical results.
\end{remark}

\begin{definition}\label{def-jump-size}
Let $f(\cdot) \in \mathcal{F}_n^{r, K}$, $\{\eta_k\}_{k = 0}^{K+1} \subset \{1, \ldots, n+1\}$ be the collection of change points of $f(\cdot)$, with $\eta_0 = 0$, $\eta_{K+1} = n+1$.  For any $k \in \{1, \ldots, K\}$, let $f_{[\eta_{k-1}/n, \eta_{k+1}/n)}(\cdot): [\eta_{k-1}/n, \eta_{k+1}/n) \to \mathbb{R}$ be the restriction of $f(\cdot)$ on $[\eta_{k-1}/n, \eta_{k+1}/n)$.  Define the reparameterisation of $f_{[\eta_{k-1}/n, \eta_{k+1}/n)}(\cdot)$ as
	\begin{equation}\label{eq-repara}
		f(x) = \begin{cases}
 			\sum_{l = 0}^r a_{k, l} (x-\eta_k/n)^l, & x \in [\eta_{k-1}/n, \eta_k/n), \\
 			\sum_{l = 0}^r b_{k, l} (x-\eta_k/n)^l, & x \in [\eta_k/n, \eta_{k+1}/n),
 		\end{cases}
	\end{equation}
	where $\{a_{k, l}, b_{k, l}\}_{l=0}^r \subset \mathbb{R}$.  
	
Associated with the change point $\eta_k$, define the effective sample size as
	\[
		\Delta_k = \min\{\eta_k - \eta_{k-1}, \, \eta_{k+1} - \eta_k\}.
	\]
	For $l \in \{0, 1, \ldots, r\}$, define
	\[
		\kappa_{k, l} = |a_{k, l} - b_{k, l}| \quad \mbox{and} \quad \rho_{k, l} = \kappa_{k, l}^2 \Delta_k^{2l + 1} n^{-2l}.
	\]
	Finally, define the signal strength associated with the change point $\eta_k$ as
	\[
		\rho_k = \max_{l = 0, \ldots, r} \rho_{k, l}.
	\]
\end{definition}

We define the jump associated with each change point of $f(\cdot) \in \mathcal{F}^{r, K}_n$ in \Cref{def-jump-size}.  The definition is based on a reparameterisation of two consecutive polynomials.  Using the notation in \Cref{def-jump-size}, due to the definition \eqref{eq-def-f-r-n-k}, $f(\cdot)$ is an at-most-degree-$r$ polynomial in each $[\eta_k/n, \eta_{k+1}/n)$, $k \in \{0, \ldots, K\}$.  This enables the reparameterisation \eqref{eq-repara}.

With the reparameterisation \eqref{eq-repara}, it is easy to see, for any change point $\eta_k$, there must exist at least one $l \in \{0, 1, \ldots, r\}$, such that $\kappa_{k, l} > 0$, thus $\rho_k > 0$ for any change point.  In addition, if $f(\cdot)$ at $\eta_k/n$ is $d$-time differentiable but not $(d+1)$-time differentiable, $d \in \{-1, \ldots, r-1\}$, then 
	\[
		\begin{cases}
 			\kappa_{k, l} = 0, & l \in \{0, \ldots, d\}, \\
 			\kappa_{k, l} > 0, & l = d+1.
 		\end{cases}
	\]
	Here we use the convention that if $f(\cdot)$ is $-1$-time differentiable at $x$, then $f(\cdot)$ at $x$ is not continuous.   	
	
Based on the definition of $\rho_k$, we remark that for each individual change point, the signal strength $\rho_k$ is associated with a certain polynomial order $l_k$ such that 
	\begin{equation}\label{eq-def-lk}
		l_k \in \argmax_{l \in \{0, \ldots, r\}}\rho_{k, l} = \argmax_{l \in \{0, \ldots, r\}}\kappa_{k, l}^2 \Delta_k^{2l + 1} n^{-2l}.
	\end{equation}
	We will come back to this after \Cref{thm}.
	

There are two key advantages of using \Cref{def-jump-size} to characterise the difference.  Firstly, we allow for a full range of smoothness at the change points.  Detecting change points in piecewise linear models was studied in \cite{fearnhead2019detecting}, but the continuity at the change points is imposed.  Our formulation covers this continuity but also allows for discontinuity.  Most importantly, we allow for each change point to have its individual smoothness level, which is $\min\{l = 0, \ldots, r: \, \kappa_{k, l} > 0\}$.

Secondly, in addition to allowing for a full range of smoothness, we also take into consideration the magnitude of coefficients change at different order.  In the piecewise linear change point detection literature, \cite{chen2020jump} considered both continuous and discontinuous cases, but assuming all the changes are either zero or of order $O(1)$ and there is only one true change point. Our formulation allows the changes and the locations of the change points to be functions of the sample size $n$, and allows for the number of change points to diverge as the sample size grows unbounded.



\subsection{Change point localisation errors}\label{sec-upper-bound}

In this section, we present our main theorem providing theoretical guarantees on the output of \Cref{algorithm:PDP}, with assumptions collected in \Cref{assume-model}.  

\begin{assumption}[Model assumptions]\label{assume-model}
Assume that the data $\{y_i\}_{i = 1}^n$ are generated from \eqref{eq-y-intro}, where $f(\cdot)$ belongs to $\mathcal{F}_n^{r, K}$ defined in \eqref{eq-def-f-r-n-k} and $\varepsilon_i$'s are independent zero mean sub-Gaussian random variables\footnote{We recall the definition of sub-Gaussian random variable \citep[e.g.~Definition 2.5.6 in][]{vershynin2018high}.  We denote $\|\cdot\|_{\psi_2}$ as the sub-Gaussian or Orlicz-$\psi_2$ norm.  For any random variable $X$, let  
	\[
		\|X\|_{\psi_2} = \inf\left\{t > 0:\, \mathbb{E}\left\{\exp(X^2/t^2)\right\} \leq 2\right\}.
	\]}  with $\max_{i = 1}^n \|\varepsilon_i\|_{\psi_2}\leq \sigma^2$. 

We denote the collection of all change points of $f(\cdot)$ to be $\{\eta_1, \ldots, \eta_K\}$, satisfying
	\[
		\Delta = \min_{k \in \{1, \ldots, K+1\}} (\eta_k - \eta_{k-1}) = \min_{k \in \{1, \ldots, K\}} \Delta_k > 0,
	\]
	where $\eta_0 = 1$, $\eta_{K+1} = n+1$ and $\{\Delta_k\}_{k = 1}^K$ are defined in \Cref{def-jump-size}.
	
In addition, for any $k \in \{1, \ldots, K\}$, let the minimal signal strength parameter be 
	\[
		\rho = \min_{k = 1, \ldots, K} \rho_k > 0,
	\]
	where $\rho_k$ is defined in \Cref{def-jump-size}.
\end{assumption}

The problem now is completely characterised by the sample size $n$, the maximum degree $r$, the number of change points $K$, the upper bound of the fluctuations $\sigma$, the effective sample sizes $\{\Delta_k\}$ and the signal strengths $\{\rho_k\}$.  In this paper, we allow the maximum degree $r$ to be arbitrary but fixed, i.e.~not a function of the sample size $n$.  We allow the number of change points $K$ and the fluctuation bound $\sigma$ to diverge, the ratios of the effective sample size to the total sample size $\{\Delta_k/n\}$ and the active jump sizes $\{\kappa_k\}$ to vanish, as the sample size grows unbounded.  



\begin{theorem}\label{thm}
Let data $\{y_i\}_{i = 1}^n$ satisfy \Cref{assume-model}.  Let $\{\widetilde{\eta}_k\}_{k = 1}^{\widehat{K}}$ and $\{\widehat{\eta}_k\}_{k = 1}^{\widehat{K}}$ be the initial estimators and final estimators of \Cref{algorithm:PDP}, with inputs $\{y_i\}_{i = 1}^n$ and tuning parameter $\lambda$.    Assume that 
	\begin{equation}\label{eq-snr}
		\lambda = c_{\mathrm{noise}} K \sigma^2 \log(n) \quad \mbox{and} \quad \rho \geq c_{\mathrm{signal}} \lambda.
	\end{equation}
	
For each $k \in \{1, \ldots, K\}$, let 
	\begin{equation}\label{eq-Sk-set}
		\mathcal{S}_k = \{l = 0, \ldots, r: \, \rho_{k, l} \geq c_{\mathrm{signal}} \lambda\}.
	\end{equation}
	In addition, let
	\begin{equation}\label{eq-active-rk-kk}
		r_k = \min \left\{\argmin_{l \in \mathcal{S}_k} \left(\frac{\sigma^2 \log(n)}{\rho_{k, l}}\right)^{1/(2l+1)}\right\}  \quad \mbox{and} \quad \kappa_k = \kappa_{k, r_k}.
	\end{equation}

We have that
	\[
		\mathbb{P}\{\mathcal{E}\} \geq 1 - n^{-c_{\mathrm{prob}}},
	\]
	where
	\begin{align*}
		\mathcal{E} = \Bigg\{& \widehat{K} = K, \;\; |\widetilde{\eta}_k - \eta_k| \leq c_{\mathrm{error}} \left\{\frac{K n^{2r_k} \sigma^2 \log(n)}{\kappa_k^2}\right\}^{1/(2r_k+1)}, \\
		&\mbox{and} \;\; |\widehat{\eta}_k - \eta_k| \leq c_{\mathrm{error}} \left\{\frac{n^{2r_k} \sigma^2 \log(n)}{\kappa_k^2}\right\}^{1/(2r_k+1)}, \;\; \forall k \in \{1, \ldots, K\}\Bigg\}.
	\end{align*}
	The constants $c_{\mathrm{prob}}$, $c_{\mathrm{noise}}$, $c_{\mathrm{signal}}$ and $c_{\mathrm{error}} > 0$ are all absolute constants.
\end{theorem}

\begin{remark}[Tracking constants]
All the absolute constants $c_{\mathrm{prob}}$, $c_{\mathrm{noise}}$, $c_{\mathrm{signal}}$, $c_{\mathrm{error}}$ can be tracked in the proof, although we do not claim the optimality of the constants thereof.  The hierarchy of the constants are as follows.  

We first determine the constant $c_{\mathrm{prob}} > 0$, which only depends on the maximum degree $r$.  Given $c_{\mathrm{prob}}$, we can determine $c_{\mathrm{noise}}$, which only depends on $c_{\mathrm{prob}}$.  With $c_{\mathrm{prob}}$ and $c_{\mathrm{noise}}$ at hand, we can determine $c_{\mathrm{signal}} > 0$.  Lastly, the constant $c_{\mathrm{error}} > 0$ depends on $c_{\mathrm{signal}}, c_{\mathrm{noise}}$ and $c_{\mathrm{prob}}$.  We note that the larger $c_{\mathrm{signal}}$ is, the smaller $c_{\mathrm{error}}$ is.  
\end{remark}

\begin{remark}[The choice of $\lambda$]
The theoretical result relies on a choice of $\lambda$ detailed in \eqref{eq-snr}, which is a function of unknown parameters $K$ and $\sigma$, in addition to an unknown quantity $c_{\mathrm{noise}}$.  In practice, we do not recommend to estimate $K$ and $\sigma$, separately, due to the involvement of $c_{\mathrm{noise}}$.  One can adopt a data-driven method for tuning parameter selection \citep[e.g.][]{rinaldo2021localizing}.
\end{remark}

To understand \Cref{thm}, we conduct discussions in the following aspects: (1) how to understand the localisation rates; (2) how to understand the definitions of $\{r_k\}$; and (3) how to understand the the signal strength condition in \eqref{eq-snr}.  We conclude this discussion with piecewise linear models as examples.

\medskip
\noindent \textbf{The localisation rates.}

From \Cref{thm} we can see that the final estimators $\{\widehat{\eta}_k\}_{k = 1}^{\widehat{K}}$ improve upon the initial estimators $\{\widetilde{\eta}_k\}_{k = 1}^{\widehat{K}}$, by getting rid of $K$, the dependence on the number of change points, in their localisation error upper bounds.  It is possible that this $K$ term is actually an artefact of our current proof, and we might not need to update our initial estimators further.  See \Cref{sec-discussion} for more on this issue.  However, with our current proof technique we do need the second step update to obtain the improved localisation error bound.

As for each individual change point $\eta_k$, $k \in \{1, \ldots, K\}$, the localisation errors are 
	\begin{equation}\label{eq-local-rate}
		|\widehat{\eta}_k - \eta_k| \lesssim |\widetilde{\eta}_k - \eta_k| \lesssim \left\{\frac{Kn^{2r_k}\sigma^2 \log(n)}{\kappa_k^2}\right\}^{1/(2r_k+1)}.
	\end{equation}
	Due to the definition of $r_k$ and the condition \eqref{eq-snr}, it holds that 
	\[
		\eqref{eq-local-rate} = \left\{\frac{K\Delta^{2r_k + 1} \sigma^2 \log(n)}{\rho_{k, r_k}}\right\}^{1/(2r_k+1)} \lesssim \Delta\left\{\frac{K \sigma^2 \log(n)}{\lambda}\right\} ^{1/(2r_k+1)} \lesssim \Delta.
	\]
	With properly chosen constants, we ensure that in the event $\mathcal{E}$, we have that $|\widetilde{\eta}_k - \eta_k| < c\Delta$, $0 < c < 1/2$.  This guarantees that in the second step in \Cref{algorithm:PDP}, each interval $[s_k, e_k)$ contains one and only one true change point.  
	
\medskip
\noindent \textbf{The definitions of $r_k$.}

For each $k \in \{1, \ldots, K\}$, the final localisation rates are functions of $r_k$, which is one of the polynomial orders in the set $\mathcal{S}_k$.  As defined in \eqref{eq-active-rk-kk}, the choice of $r_k$ minimises the term 
	\begin{equation}\label{eq-any-bound}
		\left(\frac{\sigma^2 \log(n)}{\rho_{k, l}}\right)^{1/(2l+1)}, 
	\end{equation}
	for any $l \in \mathcal{S}_k$.  In fact, it can be seen from the proofs, for $l \in \mathcal{S}_k$ as defined in \eqref{eq-Sk-set}, the term \eqref{eq-any-bound} can serve as an upper bound in the localisation error rates.  Due to the definition of $r_k$ in \eqref{eq-active-rk-kk}, we see that our choice of $r_k$ ensures that the localisation rate is the sharpest.  If the minimiser is not unique, we choose $r_k$ to be the smallest element to guarantee the uniqueness in definition.  However, we remark any choice would unveil the final rate.
	
Recall that in \Cref{sec-character} after we present \Cref{def-jump-size}, we mentioned that the individual signal strength is associated with a certain polynomial order $l_k$, defined in \eqref{eq-def-lk}.  The choice of $r_k$ in \eqref{eq-active-rk-kk} is not necessarily the same as $l_k$, but it holds that $\{r_k, l_k\}\subset \mathcal{S}_k$.  If there is only one polynomial order whose the signal strength is large enough, i.e.~$|\mathcal{S}_k| = 1$, then $r_k = l_k$; otherwise they are not necessarily the same.
	
\medskip
\noindent \textbf{The signal strength condition \eqref{eq-snr}.}

Recalling the definition of $\rho$, the condition \eqref{eq-snr} requires that
	\begin{equation}\label{eq-snr-detailed}
		\min_{k = 1, \ldots, K} \max_{l = 0, \ldots, r} \rho_{k, l} = \min_{k = 1, \ldots, K} \max_{l = 0, \ldots, r} \kappa_{k, l}^2 \Delta_k^{2l + 1} n^{-2l} \gtrsim K \sigma^2 \log(n).
	\end{equation}
	This is to say, at any true change point, there is at least one polynomial order, the jump associated with which has strength larger than $K \sigma^2 \log(n)$.   The signal strength $\rho_{k, l}$ is a function of the coefficient change size, as well as the corresponding order.  
	
\medskip
\noindent \textbf{Piecewise linear models.}

Let us consider a concrete case where $K = 1$, $r = 1$ and the only change point is $\eta$.  A question that can be asked now is as follows.  
	\begin{displayquote}
		\emph{Is it easier to estimate the change point location when the underlying $f(\cdot)$ is continuous at $\eta$ or discontinuous at~$\eta$?	}
	\end{displayquote}

This question is partially answered in \cite{chen2020jump}, while assuming $\kappa \asymp \sigma \asymp 1$, and \cite{chen2020jump} argues that (in our terminology) the localisation errors for the continuous and discontinuous cases are of order $O(n^{2/3})$ and $O(1)$, respectively.  \Cref{thm} unfolds a more comprehensive picture.  We remark that \cite{chen2020jump} has also proposed a super-efficient rate $O(n^{1/2})$ for the continuous case.  We will provide more discussions with respect to that in \Cref{sec-special-case}.

For piecewise linear functions, at the change points, using the notation in \Cref{def-jump-size}, there are three situations: (a) $\kappa_{1, 0} = 0 \neq \kappa_{1, 1}$, i.e.~$f(\cdot)$ is continuous but not differentiable at the change point; (b) $\kappa_{1, 1} = 0 \neq \kappa_{1, 0}$, i.e.~$f(\cdot)$ is discontinuous but the slope is unchanged at the change point; and (c) $ \kappa_{1, 0}\kappa_{1, 1} \neq 0$, i.e.~$f(\cdot)$ is discontinuous and the slopes are different before and after the change point.
	\begin{itemize}
	\item In case (a), provided that $\kappa_{1, 1}^2 \Delta_1^3 n^2 \gtrsim \sigma^2 \log(n)$, the localisation error rate is
	\begin{equation}\label{eq-upppppppp}
		n^{2/3} \left\{\frac{\sigma^2 \log(n)}{\kappa_{1, 1}^2}\right\}^{1/3}.
	\end{equation}
	\item In case (b), provided that $\kappa_{1, 0}^2 \Delta_1 \gtrsim \sigma^2 \log(n)$, the localisation error rate is	
	\[
		\frac{\sigma^2 \log(n)}{\kappa_{1, 0}^2}.
	\]
	\item In case (c), there are further sub-cases.  For readability, we assume $\Delta_1 = n/2$, i.e.~the change point is right at the middle.  
	\begin{itemize}
	\item  If
		\[
			\kappa_{1, 1}^2 n \gtrsim \sigma^2 \log(n) \gtrsim \kappa_{1, 0}^2 n,
		\]	
		then the localisation error rate is
		\[
			n^{2/3} \left\{\frac{\sigma^2 \log(n)}{\kappa_{1, 1}^2}\right\}^{1/3}.
		\]
	\item  If
		\[
			\kappa_{1, 0}^2 n \gtrsim \sigma^2 \log(n) \gtrsim \kappa_{1, 1}^2 n,
		\]	
		then the localisation error rate is
		\[
			\frac{\sigma^2 \log(n)}{\kappa_{1, 0}^2}.
		\]
	\item If 	
		\[
			\min \{\kappa_{1, 1}^2 n, \, \kappa_{1, 0}^2 n\} \gtrsim \sigma^2 \log(n),
		\]	
		then the localisation error rate is
		\[
			\min\left\{n^{2/3} \left\{\frac{\sigma^2 \log(n)}{\kappa_{1, 1}^2}\right\}^{1/3}, \frac{\sigma^2 \log(n)}{\kappa_{1, 0}^2}\right\}.
		\]
	\end{itemize}
	\end{itemize}

Back to the the question we asked above, there is no simple answer that which case is simpler and one needs to carefully consider the different rates we discussed above.  But if one assume $\kappa^{-2}\sigma^2 \log(n) \asymp 1$, the continuous and discontinuous cases yield localisation rates as $O(n^{2/3})$ and $O(1)$, respectively.

\subsection{Global lower bounds}\label{sec-lb-state}

In this section, we aim to provide global information-theoretic lower bounds to characterise the fundamental difficulties of localisation change points in the model defined in \Cref{assume-model}.   By ``global'' we mean we do not assuming knowing further continuity conditions, in contrast to \Cref{sec-special-case} in the sequel.

In the change point analysis literature, in terms of localising the change point locations, there are two aspects we are interested in.  One is the minimax lower bound on the localisation error and the other is on the signal strength.  For simplicity, in this section, we assume that $K = 1$ and $r_1 = r$, using the notation defined in \eqref{eq-active-rk-kk}.

As for these two aspects, in \Cref{thm}, we show that provided
	\begin{equation}\label{eq-snr-upper-bound}
		\frac{\kappa^2 \Delta^{2r+1}}{n^{2r}} \gtrsim K \sigma^2 \log(n),
	\end{equation}
	the output returned by \Cref{algorithm:PDP} have localisation error upper bounded by
	\[
		\left\{\frac{n^{2r} \sigma^2 \log(n)}{\kappa^2}\right\}^{1/(2r+1)}.
	\]
	In this section, we will investigate the optimality of the above results.  
	
\begin{lemma}\label{lem-lb-2}
Under \Cref{assume-model}, assume that there exists one and only one change point and $r_1 = r$.  Let $P_{\kappa, \Delta, \sigma, r, n}$ denote the joint distribution of the data.  Consider the class
	\[
		\mathcal{Q} = \left\{P_{\kappa, \Delta, \sigma, r, n}: \, \Delta < n/2,  \, \kappa^2 \Delta^{2r+1} \geq \sigma^2 n^{2r} \zeta_n\right\},
	\]
	for any diverging sequence $\{\zeta_n\}$.  Then for all $n$ large enough, it holds
	\[
		\inf_{\widehat{\eta}} \sup_{P \in \mathcal{Q}} \mathbb{E}_P(|\widehat{\eta} - \eta(P)|) \geq \left[\frac{c n^{2r} \sigma^2}{ \kappa^2}\right]^{1/(2r+1)},
	\]
	where $\eta(P)$ is the location of the change point for distribution $P$, the minimum is taken over all the measurable functions of the data, $\widehat{\eta}$ is the estimated change point and $0 < c < 1$ is an absolute constant.
\end{lemma}

\Cref{lem-lb-2} shows that the final estimators provided by \Cref{algorithm:PDP} are nearly optimal, in terms of the localisation error, save for a logarithmic factor.  We remark that in \Cref{lem-lb-2}, we consider the class of distributions with the signal strength at order $r$ satisfies the signal-to-noise ratio condition \eqref{eq-snr-detailed}, and the order $r$ is used in the localisation error lower bounds.   We leave the proof of \Cref{lem-lb-2} in the appendix, but we provide some explanations of the proofs here.

We adopt Le Cam's lemma \citep[e.g.][]{yu1997assouad} to show the lower bound, and consider two explicit distributions when applying Le Cam's lemma.  One of these two distributions is $(r-1)$-time differentiable but not $r$-time differentiable.  The other distribution is not continuous.  This construction provides us a global minimax lower bound when $r_1 = r$.  For example, in the piecewise linear models, this includes both continuous and discontinuous cases, and the corresponding lower bound is of order
	\[
		n^{2/3} \left(\frac{\sigma^2}{\kappa^2}\right)^{1/3}.
	\]
	Combining with \eqref{eq-upppppppp}, we know \Cref{thm} is optimal saving a logarithmic factor. 

\begin{lemma}\label{lem-lb-1}
Under \Cref{assume-model}, assume that there exists one and only one change point and $r_1 = r$.  Let $P_{\kappa, \Delta, \sigma, r, n}$ denote the joint distribution of the data.  For a small enough $\xi > 0$, consider the class
	\[
		\mathcal{P} = \left\{P_{\kappa, \Delta, \sigma, r, n}: \, \Delta = \min \left\{\Bigg\lfloor \left(\frac{\xi n^{2r}}{\kappa^2 \sigma^{-2} }\right)^{1/(2r+1)} \Bigg\rfloor,\, n/3 \right\} \right\}.
	\]	
	Then we have
	\[
		\inf_{\widehat{\eta}} \sup_{P \in \mathcal{P}} \mathbb{E}_P(|\widehat{\eta} - \eta(P)|) \geq cn,
	\]
	where $\eta(P)$ is the location of the change point for distribution $P$, the minimum is taken over all the measurable functions of the data, $\widehat{\eta}$ is the estimated change point and $0 < c < 1$ is an absolute constant depending on $\xi$.
\end{lemma}

\Cref{lem-lb-1} shows that, if $\kappa^2 \Delta^{2r+1}n^{-2r} \lesssim \sigma^2$, then no algorithm is guaranteed to be consistent, in the sense that
	\[
		\inf_{\widehat{\eta}} \sup_{P \in \mathcal{P}} \mathbb{E}_P\left(\frac{|\widehat{\eta} - \eta(P)|}{n}\right) \gtrsim 1.
	\]
	This means, besides the logarithmic factor, \Cref{lem-lb-1} and \Cref{thm}  leave a gap in terms of $K$.  To be specific, it remains unclear what results one would obtain if 
	\begin{equation}\label{eq-gapppp}
	\sigma^2 \lesssim \kappa^2 \Delta^{2r+1}n^{-2r} \lesssim K\sigma^2.
	\end{equation}  
	This gap only exists when we allow $K$ to diverge.   We will provide some conjectures inline with this discussion in \Cref{sec-comparisons}.

\subsection{A special case}\label{sec-special-case}
	
In \Cref{sec-lb-state} we have shown the global minimax lower bound on the localisation error.  In \cref{lem-lb-3333} below, we provide a minimax lower bound in a smaller class.
	
\begin{lemma}\label{lem-lb-3333}
Under \Cref{assume-model}, assume that there exists one and only one change point and $r_1 = r \geq 1$.  Let $P_{\kappa, \Delta, \sigma, r, n}$ denote the joint distribution of the data.  Consider the class
	\begin{align*}
		\mathcal{Q}_1 = \big\{P_{\kappa, \Delta, \sigma, r, n}: \, \Delta < n/2,  &\, \kappa^2 \Delta^{2r+1} \geq \sigma^2 n^{2r} \zeta_n\\
		& \mbox{ and } a_{1, l} = b_{1, l}, \, l = 0, \ldots, r - 1\big\},
	\end{align*}
	for any diverging sequence $\{\zeta_n\}$.  Then for all $n$ large enough, it holds
	\[
		\inf_{\widehat{\eta}} \sup_{P \in \mathcal{Q}_1} \mathbb{E}_P(|\widehat{\eta} - \eta(P)|) \geq \left[\frac{c n \sigma^2}{ \kappa^2}\right]^{1/2},
	\]
	where $\eta(P)$ is the location of the change point for distribution $P$, the minimum is taken over all the measurable functions of the data, $\widehat{\eta}$ is the estimated change point and $0 < c < 1$ is an absolute constant.
\end{lemma}

Comparing Lemmas~\ref{lem-lb-2} and \ref{lem-lb-3333}, we notice that $\mathcal{Q}_1$ the class of distributions considered in \Cref{lem-lb-3333} is strictly smaller than $\mathcal{Q}$ the class of distributions considered in \Cref{lem-lb-2}.  In $\mathcal{Q}_1$, we enforce that the underlying polynomials are $(r-1)$-time differentiable.  We leave the proof of \Cref{lem-lb-3333} in the appendix, but we highlight some key ingredients here.  We again adopt Le Cam's lemma in deriving the lower bound, but different from the construction used in the proof of \Cref{lem-lb-2}, the two explicit functions we choose are both $(r-1)$-time differentiable.  

Apparently, the localisation lower bound provided in \Cref{lem-lb-3333} is sharper than the one in \Cref{lem-lb-2}.  This is not surprising, since $\mathcal{Q}_1 \subset \mathcal{Q}$.  What is seemingly surprising is that the lower bound is not a function of any polynomial order.  This is gained by knowing the fact that $a_{1, l} = b_{1, l}$, $l \in \{0, \ldots, r-1\}$.
	
In \cite{chen2020jump}, similar results were obtained but only for the piecewise linear case.  To match this lower bound, \cite{chen2020jump} proposed a super-efficient estimator, which assumes that it is known the piecewise linear models are continuous.  The super-efficient estimator is essentially a penalised estimator, which forces the intercept estimators to equal, if their difference is not very large.  One can straightforwardly extend the idea there to the class $a_{k, l} = b_{k, l}$, $l \in \{0, \ldots, r_k-1\}$, but $a_{k, r_k} \neq b_{k, r_k}$, for any $k \in \{1, \ldots, K\}$.  Enforcing the corresponding polynomial coefficient estimators to equal before and after each change point estimator, knowing the exact smoothness at every individual true change point, will prompt a localisation error of order detailed in \Cref{lem-lb-3333}.  We would refrain from proposing such an effort, since in our paper, we allow for multiple change points and allow for individual smoothness levels.  This will end up with $\sum_{k = 1}^K r_k$ more tuning parameters.

\section{Discussions} \label{sec-discussion}

In this paper, we investigate the change point localisation in piecewise polynomial signals.  We allow for a general framework and provide individual localisation error, associated with the individual smoothness at each change point.  A two-step algorithm consisting of solving a minimal partitioning problem and an updating step is proposed.  The outputs are shown to be nearly-optimal, supported by the information-theoretic lower bounds.  To conclude this paper, we discuss some unresolved aspects of our work while comparing our results to some particularly relevant existing literature.  Readers who are less familiar with the change point literature may safely skip this section.

\subsection{Comparisons with \cite{wang2020univariate}}\label{sec-comparisons}

\cite{wang2020univariate} studied change point localisation in piecewise constant signals.  They studied the $\ell_0$-penalised least squares method and proved that it is nearly minimax optimal in terms of both the signal strength condition and the localisation error.  In contrast, with our proof technique, we have been able to generalise this result for higher degree polynomials up to a factor depending on $K$, the number of true change points.  This can be seen in our change point localisation error bound of our initial estimators as provided in \Cref{thm} and also in our required signal strength condition in \eqref{eq-snr-detailed}.  In our paper, with general degree polynomials, the localisation near-optimality is secured via an extra updating step, and a gap remains in the upper and lower bounds for our required signal strength condition. This gap is not present if $K$ is assumed to be $O(1)$ but is present if it is allowed to diverge. 

We explain why the proof in~\cite{wang2020univariate} could not be fully generalised to our setting.  Recall the definition of $H(v, I)$ in \eqref{eq-def-G-loss} denoting a residual sums of squares term.  In our analysis, a crucial role is played by the term
	\[
		Q\{\mathbb{E}(y); I_1, I_2\} = H\{\mathbb{E}(y), I_1 \cup I_2\} - H\{\mathbb{E}(y), I_1\} - H\{\mathbb{E}(y), I_2\},
	\]
	where $I_1, I_2$ are two contiguous intervals of $\{1, \ldots, n\}$.   Ideally, one needs to be able to upper and lower bound $Q\{\mathbb{E}(y); I_1, I_2\}$ when $y$ is defined in \eqref{eq-y-intro}, and its corresponding $f(\cdot)$ is a degree-$r$ polynomial on $I_1$ and another degree-$r$ polynomial on $I_2$.   In the case of $r = 0$, i.e.~in the piecewise constant case, one can write an exact expression 
	\[
		Q\{\mathbb{E}(y); I_1, I_2\} = \frac{|I_1||I_2|}{|I_1| + |I_2|} \left(|I_1|^{-1}\sum_{i \in I_1} \mathbb{E}(y_i) - |I_2|^{-1}\sum_{i \in I_2} \mathbb{E}(y_i)\right)^2.
	\]
	In addition, it holds that
	\[
		\frac{\min\{|I_1|, \, |I_2|\}}{2} = \frac{|I_1||I_2|}{2 \max\{|I_1|, \, |I_2|\}} \leq \frac{|I_1||I_2|}{|I_1| + |I_2|} \leq \min\{|I_1|, \, |I_2|\}.
	\]
	Therefore, it follows that 
	\begin{equation}\label{eq-chain-inequ-disc}
		\frac{1}{2} \min\{|I_1|, \, |I_2|\} \kappa^2 \leq Q\{\mathbb{E}(y); I_1, I_2\} \leq \min\{|I_1|, \, |I_2|\} \kappa^2,
	\end{equation}
	where $\kappa$ represents the absolute difference between the values of $\mathbb{E}(y_i)$, $i \in I_1$ and $i \in I_2$.

For general $r$, by adopting an elegant result in \cite{shen2020phase}, one can actually generalise \eqref{eq-chain-inequ-disc} to obtain that
	\begin{equation}\label{eq-Q-lack-upper}
		C_1 \frac{\min\{|I_1|^{2r + 1}, \, |I_2|^{2r + 1}\}}{n^{2r}}\kappa^2 \leq Q\{\mathbb{E}(y); I_1, I_2\} \leq C_2 \frac{\min\{|I_1|^{2r+1}, \, |I_2|^{2r+1}\}}{n^{2r}} \kappa^2,
	\end{equation}
	where $0 < C_1 < C_2$ are two absolute constants, and $\kappa$ is the absolute difference of the $r$th degree coefficients of $\mathbb{E}(y)$ on $I_1$ and $I_2$.  However, the problem is that the constants $C_1$ and $C_2$ are not explicit.  We can only show the existence of such constants.  Even if we can track these two constants down, in order to be able to generalise the argument of \cite{wang2020univariate}, we would still need to show that $C_1$ and $C_2$ are close enough.  At this moment, it is not clear to us how to resolve this issue.  We can only conjecture that for all $r \in \mathbb{N}$, the $\ell_0$-penalised least squares method would itself be nearly optimal in terms of both the signal strength condition and the localisation error, and our second step update would not be needed. From a practical point of view, our second step can be done in $O(n)$ time, which is negligible compared to the $O(n^3)$ time required to solve the penalised least squares.  The computational overhead of our second step is thus minor. 

\subsection{Comparisons with \cite{fearnhead2019detecting}}\label{sec-comparisons2}
 
\cite{fearnhead2019detecting} showed that penalised least squares method for change point localisation works well for piecewise linear signals.  This work inspired us to investigate piecewise polynomial signals of higher degrees.  Even in the piecewise linear case, there are some differences between our work and \cite{fearnhead2019detecting}.  The algorithm provided in \cite{fearnhead2019detecting} can be seen as solving a variant of the penalised least squares problem mentioned in this paper.  In fact, the dynamic programming algorithm mentioned in \cite{fearnhead2019detecting} appears to be more sophisticated than what would be required to solve our problem.  It is because the algorithm in~\cite{fearnhead2019detecting} is tailored specifically for continuous piecewise linear functions.  Maintaining continuity makes the dynamic programming algorithm more involved. Translated to our notation, \cite{fearnhead2019detecting} assumes $r_k = 1$, for all $k \in \{1, \ldots, K\}$.  Our formulation is more general than \cite{fearnhead2019detecting} as we do not impose continuity or any kind of smoothness at the change points.  Our estimator adapts near-optimally to the level of smoothness at the change points. The theoretical results studied in \cite{fearnhead2019detecting} are under the conditions $K, \sigma \asymp 1$.  Under these conditions, translated to our notation, their results read, provided that $(\kappa/n)^2 \Delta^3 \gtrsim \log(n)$, the localisation error is $\log^{1/3}(n) (n/\kappa)^{2/3}$.  Both are consistent with the results we have obtained in this paper.  

We would like to emphasis that when the underlying functions are indeed continuous at the change points, our estimators may be discontinuous, but our estimators will be very close to continuous functions; in the sense that $\mathrm{MSE}(\widehat{\theta}, \theta^*) \to 0$, $n \to \infty$, where $\widehat{\theta} = (\widehat{f}(i/n), i = 1, \ldots, n)^{\top}$ and $\theta^* = (f^*(i/n), i = 1, \ldots, n)^{\top}$; see \cite{shen2020phase}.

\subsection{Comparisons with  \cite{raimondo1998minimax}}\label{sec-other-comp}

\cite{raimondo1998minimax} studied the minimax rates of change point localisation in a nonparametric setting.  The main focus there is how the localisation errors' minimax rates change with $\alpha$, the degree of discontinuity in a H\"{o}lder sense.  Due to the nonparametric essence, the class of functions considered in \cite{raimondo1998minimax} is more general than the piecewise polynomial class we discuss here.  However, the measures of regularity $r_k$'s we have defined in \Cref{def-jump-size} are similar as the parameter $\alpha$ in \cite{raimondo1998minimax}, if we only consider polynomials.  Having drawn this connection, translated into our notation, \cite{raimondo1998minimax} in fact shows that the localisation error's minimax lower bound is of order 
	\[
		\left\{n^{2r}\log^{\eta}(n)\right\}^{1/(2r+1)}, \quad \forall \eta > 1.
	\]
	This is a lower bound for a larger class of functions than ours, but the dependence on $n$ is the same as ours up to a poly-logarithmic factor. In general, the larger the class is, the smaller the minimax lower bound is.  Since \cite{raimondo1998minimax} assumes all the other parameter to be of order $O(1)$, our minimax lower bounds add value as they are in terms of all the relevant problem parameters and not just the sample size $n$.

\subsection{Why not just differencing the sequences}

In this paper, we are dealing with piecewise polynomials with general order $r$. We noticed that in practice, some practitioners tend to difference the sequences $r$ times, wishing to obtain piecewise constant signals, and then conduct change point detection methods on the resulting differenced sequence.  This is in fact not an effective method if the goal is to detect change points. 

We use piecewise linear models as concrete examples, assuming we have
	\[
		f(x) = \begin{cases}
			a_0 + a_1 (x-\eta/n), & x \in [0, \eta/n), \\
 			a_0 + b_1 (x-\eta/n), & x \in [\eta/n, 1),
 		\end{cases}
	\]
	where $a_0, a_1, b_1 \in \mathbb{R}$ and $a_1 \neq b_1$.  As we have shown, the global and constrained minimax lower bounds on this problem are 
	\begin{equation}\label{eq-wwwwwww}
		\left(\frac{\sigma^2 }{(a_1 - b_1)^2}\right)^{1/3} n^{2/3} \quad \mbox{and} \quad \left(\frac{\sigma^2 }{(a_1 - b_1)^2}\right)^{1/2} n^{1/2}, 
	\end{equation}
	respectively.
	
If we now take differences, then we work under a new model 
	\[
		g(x) =  \begin{cases}
			a_1/n, & x \in [0, \eta/n), \\
 			b_1/n, & x \in [\eta/n, 1).
 		\end{cases}
	\]	
	This is now a piecewise constant case, the localisation error lower bound is now of order
	\begin{equation}\label{eq-wwwwvfvv}
		\frac{\sigma^2}{(a_1/n-b_1/n)^2} = \frac{n^2\sigma^2}{(a_1-b_1)^2},
	\end{equation}
	provided that the signal strength is still strong enough.  (The differenced sequence is no longer independent, but weakly dependent.  Therefore the variance parameter is inflated by a constant.)
	
Comparing the rates in \eqref{eq-wwwwvfvv} and \eqref{eq-wwwwvfvv}, we show that it is not always a good idea to difference the polynomial sequences.

\section{Numerical experiments} \label{sec-numeric}

In this section, we conduct extensive numerical experiments, based on piecewise quadratic functions.

\medskip
\noindent \textbf{Evaluation measurements.}  Letting $\{\widehat{\eta}_k\}_{k = 1}^{\widehat{K}}$ and $\{\eta_k\}_{k = 1}^K$ be estimated and true change points, respectively, we evaluate the performances of $\{\widehat{\eta}_k\}_{k = 1}^{\widehat{K}}$ using $|\widehat{K} - K|$ and the scaled Hausdorff distance, i.e.
	\begin{align*}
    	d_{\mathrm{H}} = \frac{1}{n}\max\Big\{\max_{j = 0, \dots \widehat{K}+1}\min_{k = 0, \dots, K+1}|\widehat{\eta}_{j} - \eta_k|, \, \max_{k = 0, \dots , K+1} \min_{j = 0, \dots, \widehat{K}+1} |\widehat{\eta}_{j} - \eta_k| \Big\},
	\end{align*}
	where $\widehat{\eta}_0 = \eta_0 = 0$ and $\widehat{\eta}_{\widehat{K}+1} = \eta_{K+1} = n$.

\medskip
\noindent \textbf{Tuning parameter selection.}   The only tuning parameter $\lambda$ is selected via the cross-validation method \cite{rinaldo2021localizing}.  To be specific, we first divide the sample into training and validation sets according to odd and even indices.  For each possible values of $\lambda$ considered, the initial estimator $\widehat{\Pi}$ is obtained based on the training set. On the validation set, for each $I \in \widehat{\Pi}$, we obtain $\hat{y}_{I} = P_{I}y_{I}$ and compute the validation loss $\sum_{t \mod 2 \equiv 0}(\hat{y}_t - y_t)^2$. Finally, we select the $\lambda$ which minimises the validation loss.

\medskip
\noindent \textbf{General settings.}  With the notation in \Cref{def-jump-size}, $f(x)$ on the interval $[\eta_{k}/n, \eta_{k+1}/n)$, for $k \in \{1, \dots, K-1\}$, are represented by different polynomials $\{(x - \eta_k/n)^l\}_{l = 0}^r$ and $\{(x - \eta_{k+1}/n)^l\}_{l = 0}^r$, with coefficients $\{b_{k,l}\}_{l = 0}^r$ and $\{a_{k+1,l}\}_{l=0}^r$ respectively.  To be specific, for $r = 2$, we have
	\[
		\begin{pmatrix}
		a_{k+1,0} \\
		a_{k+1,1} \\
		a_{k+1,2}
		\end{pmatrix} = \begin{pmatrix}
			1 & \frac{\eta_{k+1} - \eta_{k}}{n} & \big(\frac{\eta_{k+1} - \eta_{k}}{n}\big)^2 \\
			0 & 1 & 2\frac{\eta_{k+1} - \eta_{k}}{n} \\ 
			0 & 0 & 1
		\end{pmatrix} \begin{pmatrix}
			b_{k,0} \\
			b_{k,1} \\ 
			b_{k,2}
		\end{pmatrix},
	\]
	The piecewise polynomials $f(x)$ can therefore be parameterised by the degree $r$, the change points $\{\eta_k\}_{k = 1}^K$, the sample size $n$, the coefficients $\{a_{1,l}\}_{l = 0}^r$ for the first segment, the jump sizes $\{\kappa_{k,l}\}_{l = 0}^r$ for $k = 1, \dots, K$, and $\sigma^2$ which quantifies the tail behavior of error terms.

For each setting below, we simulate $100$ repetitions and fix $K = 2$. Fixing the effects of $K$ and $\sigma^2$, the localisation errors shown in \Cref{thm} can be regarded as an interplay among $\Delta$, $r_k$, $n$ and $\rho_{k, r_k}$; see \eqref{eq-local-rate}.

\subsection{Scenario 1: The effects of $r_k$ and $n$}

In this scenario, we investigate the roles of $r_k$ and $n$, with equally-spaced change points.  We fix the polynomial coefficients for the first segment as $\{a_{1,l}\}_{l = 0}^2 = \{-2, 2, 9\}$, and the jumps at the first change point as $\{\kappa_{1,l}\}_{l = 0}^2 = \{3, 9, -27\}$.  We thus have $r_1 = 0$ and $\rho_{1, r_1}/n = 3$ fixed.  We further let $n \in \{150, 300, 450\}$, and for the second change point, let $\{\kappa_{2,l}\}_{l = 0}^2$ vary according to (a): $\{-3, 9, -27\}$, (b): $\{-3, 9, 0\}$, (c): $\{-3, 0, -27\}$, (d): $\{-3, 0, 0\}$, (e): $\{0, 9, -27\}$, (f): $\{0, 9, 0\}$ and (g): $\{0, 0, -27\}$.  By varying $n$ and $\{\kappa_{2,l}\}_{l = 0}^2$, the ratio $\rho_{2, r_2}/n = 3$ is fixed, but $r_2 \in \{0, 1, 2\}$ varies. Under the above settings, we have the localisation errors of $\eta_2$ dominates that of $\eta_1$, and \eqref{eq-local-rate} for $\eta_2$ reduces to
	\[
	    \frac{|\widehat{\eta}_2 - \eta_2|}{n} \lesssim \frac{1}{3}\bigg\{ \frac{K\sigma^2\log(n)}{\rho_{2, r_2}} \bigg\}^{1/(2r_2 + 1)} = \frac{1}{3}\bigg\{ \frac{2\log(n)}{3n} \bigg\}^{1/(2r_2 + 1)},
	\]
	which suggests the following.  First, fixing $n$, the localisation error increases as $r_2$ increases; second, fixing $r_2$, the localisation error decreases as $n$ increases.  This is supported by the results collected in \Cref{tab_sce1_1}.  For fixed $n$ and $\Delta$, our method performs similarly for Cases (a)-(d) with the same $r_2 = 0$, and the performances deteriorate as $r_2$ increases.  In each case, the performances improve as $n$ increases.  We would like to mention that, when $r_2 = 2$, with a much larger signal strength, we can show a similarly good performance as that in the case $r_2 = 0$.

\begin{table}[!ht]
\centering
\caption{Experiment results of Scenario 1.  Each cell is in the form of mean (std.~error) over $100$ repetitions.  }
\label{tab_sce1_1}
\begin{tabular}{c|ccccccc}
\hline
$\{\kappa_{2,l}\}_{l = 0}^2$ & (a)   & (b)   & (c)   & (d)   & (e)   & (f)   & (g)   \\
$r_2$           & 0     & 0     & 0     & 0     & 1     & 1     & 2     \\ \hline
               & \multicolumn{7}{c}{$n = 150$}          \\
$|\widehat{K} - K|$        & 0.42 & 0.40 & 0.44 & 0.42 & 0.96 & 0.99 & 1.02 \\
	& (0.083) & (0.083) & (0.086) & (0.085) & (0.051) & (0.044) & (0.051) \\
$d_{\mathrm{H}}$ of $\widetilde{\eta}$        & 0.057 & 0.049 & 0.055 & 0.049 & 0.274 & 0.289 & 0.292 \\
	& (0.008) & (0.006) & (0.007) & (0.006) & (0.010) & (0.009) & (0.008) \\
$d_{\mathrm{H}}$ of $\widehat{\eta}$         & 0.049 & 0.042 & 0.048 & 0.041 & 0.270 & 0.289 & 0.290 \\
	& (0.008) & (0.006) & (0.007) & (0.006) & (0.010) & (0.008) & (0.008) \\ \hline
               & \multicolumn{7}{c}{$n = 300$}          \\
$|\widehat{K} - K|$        & 0.20 & 0.20 & 0.20 & 0.20 & 0.72 & 0.95 & 0.82 \\
	& (0.047) & (0.047) & (0.047) & (0.047) & (0.055) & (0.069) & (0.046) \\
$d_{\mathrm{H}}$ of $\widetilde{\eta}$        & 0.025 & 0.025 & 0.025 & 0.025 & 0.252 & 0.296 & 0.281 \\
	& (0.004) & (0.004) & (0.004) & (0.004) & (0.012) & (0.008) & (0.009) \\
$d_{\mathrm{H}}$ of $\widehat{\eta}$         & 0.022 & 0.022 & 0.022 & 0.022 & 0.248 & 0.290 & 0.274 \\
	& (0.004) & (0.004) & (0.004) & (0.004) & (0.012) & (0.009) & (0.010) \\
 \hline
               & \multicolumn{7}{c}{$n = 450$}         \\
$|\widehat{K} - K|$        & 0.13 & 0.13 & 0.13 & 0.13 & 0.72 & 0.88 & 0.73 \\
	& (0.034) & (0.034) & (0.034) & (0.034) & (0.070) & (0.033) & (0.053) \\
$d_{\mathrm{H}}$ of $\widetilde{\eta}$        & 0.017 & 0.017 & 0.017 & 0.017 & 0.230 & 0.313 & 0.257 \\
	& (0.004) & (0.004) & (0.004) & (0.004) & (0.012) & (0.006) & (0.011) \\
$d_{\mathrm{H}}$ of $\widehat{\eta}$         & 0.014 & 0.014 & 0.014 & 0.014 & 0.231 & 0.310 & 0.256 \\
	& (0.003) & (0.003) & (0.003) & (0.003) & (0.011) & (0.006) & (0.010) \\ \hline
\end{tabular}
\end{table}

\subsection{Scenario 2: The effect of $\Delta$}

In this scenario, we vary the minimal spacing $\Delta$ and consider un-balanced change points.  We let $n = 450$, $r = 2$, the polynomial coefficients of the first segment be $\{a_{1,l}\}_{l = 0}^2 = \{-2, 2, 9\}$, the jump sizes at the first change point be $\{\kappa_{1,l}\}_{l = 0}^2 = \{3, 9, -27\}$ and the jump sizes at the second change point be $\{\kappa_{2,l}\}_{l = 0}^2 = \{-3, 9, -27\}$.  

We consider the following six cases of the true change points $\{\eta_k\}_{k=1}^2$ as: (a): $\{50, 100\}$, (b): $\{50, 150\}$, (c): $\{50, 400\}$, (d): $\{100, 350\}$, (e): $\{150, 300\}$ and (f): $\{200, 250\}$.  The lengths of three intervals separated by true change points are (a): $(50, 50, 350)$, (b): $(50, 100, 300)$, (c): $(50, 350, 50)$, (d): $(100, 250, 100)$, (e): $(150, 150, 150)$ and (f): $(200, 50, 200)$.   By fixing $n$ and the jump sizes, we ensure $r_1 = r_2 = 0$ is unchanged for all cases.  

The results collected in \Cref{tab_sce2} show that, keeping other factors unchanged, the more balanced the locations of change points are, the better the performance of our estimator.

\begin{table}[ht]
\centering
\label{tab_sce2}
\caption{Experiment results of Scenario 2.  Each cell is in the form of mean (std.~error) over $100$ repetitions. }
\begin{tabular}{c|cccccc}
\hline
$\{\eta_k\}_{k = 0}^2$ & (a)   & (b)   & (c)   & (d)   & (e)   & (f)\\ \hline
$|\widehat{K} - K|$        & \bf{0.12} & 0.26 & 0.20 & 0.17 & 0.13 & 0.26 \\
	& (0.038) & (0.116) & (0.055) & (0.043) & (0.034) & (0.066) \\
$d_\mathrm{H}$ of $\widetilde{\eta}$        & 0.025 & 0.022 & 0.030 & 0.019 & \bf{0.017} & 0.030 \\
	& (0.006) & (0.005) & (0.008) & (0.004) & (0.004) &  (0.005) \\
$d_\mathrm{H}$ of $\widehat{\eta}$         & 0.024 & 0.020 & 0.034 & 0.018 & \bf{0.014} & 0.027 \\
	& (0.007) & (0.005) & (0.008) & (0.005) & (0.004) & (0.005) \\
 \hline
\end{tabular}
\end{table}

\appendix

\section{Summary}
We include all the proofs in the Appendices.  Some preparatory results are provided in \Cref{sec-proofs}.  \Cref{sec-proofs-main} contains the proof of \Cref{thm}.  The lower bounds results Lemmas~\ref{lem-lb-2} and \ref{lem-lb-1} are proved in \Cref{sec-proofs-lb}.

\section{Preparatory Results}\label{sec-proofs}

The following notation will be used throughout the proofs. For any $I = \{s, \ldots, e\} \subset \{1, \ldots, n\}$, recall the projection matrix $P_I$ defined in \eqref{eq-P-I-definition} using matrix $U_{I, r}$ defined in \eqref{eq-U-i-r}.  We recall the notation
	\[
		H(v, I) = \|v_I\|^2 - \|P_I v_I\|^2 = \|v_I - P_I v_I\|^2,
	\]
	for any vector $v \in \mathbb{R}^n$, where $v_I = (v_i, i \in I)^{\top} \in \mathbb{R}^{|I|}$.

For any contiguous intervals $I, J \subset \{1, \ldots, n\}$ and for any vector $v \in \mathbb{R}^n$, define
	\[
		Q(v; I, J) = H(v, I \cup J) - H(v, I) - H(v, J) = \|P_I v_I\|^2 + \|P_J v_J\|^2 - \|P_{I \cup J} v_{I \cup J}\|^2.
	\]
	
\begin{lemma}\label{lem:1}
Let $I$ be any nonempty interval subset of $\{1, \ldots, n\}$.  For any $k \in \{1, \ldots, |I|\}$ and any partition of $I$, $I = \cup_{l = 1}^k I_l$, satisfying $I_s \cap I_u = \emptyset$, for any $s, u \in \{1, \ldots, k\}$, $s \neq u$.  It holds for any vector $v \in R^n$ that
	\[
		H(v, I) \geq \sum_{l = 1}^k H(v, I_l).
	\]
\end{lemma}

\begin{proof}
The claims holds due to that 
	\[
		H(v, I) = \|v_I - P_I v_I\|^2 = \sum_{l = 1}^k \|v_{I_l} - (P_I v_I)_{I_l}\|^2 \geq \sum_{l = 1}^k \|v_{I_l} - P_{I_l} v_{I_l}\|^2.
	\]
\end{proof}

\begin{lemma}\label{lem:2}
Let $y \in \mathbb{R}^n$ satisfy $y = \theta + \varepsilon$ and $\mathbb{E}(y) = \theta$.  Let $I, J$ be two contiguous interval subsets of $\{1, \ldots, n\}$.   It holds that
	\[
		Q(y; I, J) \geq \big|\sqrt{Q(\theta; I, J)} - \sqrt{Q(\varepsilon; I, J)}\big|^2.
	\]
\end{lemma}

\begin{proof}
First observe that $Q(y; I,J)$ is a quadratic form in $y$.  Moreover, it is a positive semidefinite quadratic form as $Q(y; I, J) \geq 0$ for all $y \in \mathbb{R}^n$ by \Cref{lem:1}.  Therefore, we can write $Q(y; I, J) = y^{\top} A y$, for a positive semidefinite matrix $A \in \mathbb{R}^{n \times n}$.  Denoting $A^{1/2}$ as the square root matrix of $A$, satisfying $A^{1/2} A^{1/2} = A$, we can write $Q(y; I, J) = \|A^{1/2} y\|^2$.   It then holds that
	\begin{align*}
		\sqrt{Q(y;I,J)} =& \|A^{1/2} (\theta + \varepsilon)\|\\
		\geq& \max\left\{\|A^{1/2} \theta\| - \|A^{1/2} \varepsilon\|, \, \|A^{1/2} \varepsilon\| - \|A^{1/2} \theta\|\right\},
	\end{align*}
	which leads to the final claim.
\end{proof}

\begin{lemma}[Lemma~E.1 in \cite{shen2020phase}]\label{lem-shen2020}
There exists an absolute constant $c_{\mathrm{poly}}$ depending only on $r$ such that for any integers 
	\begin{equation}\label{eq-lem-e1-cond}
	n \geq 1, \quad m \geq r+1, \quad m \leq n
	\end{equation} 
	and any real sequence $\{a_{\ell}\}_{\ell = 0}^{r}$, 
	\[
		\sum_{i = 1}^{m} \left[a_0 + a_1 \left(\frac{i}{n}\right) + \cdots + a_{r} \left(\frac{i}{n}\right)^r\right]^2 \geq c_{\mathrm{poly}} \max_{d = 1, \ldots, r} \frac{a_d^2 m^{2d+1}}{n^{2d}}.
	\]
\end{lemma}

\Cref{lem-shen2020} is a direct consequence of Lemma~E.1 in \cite{shen2020phase}.  We omit its proof here.

\begin{proposition}\label{prop:1}
Let $I = \{s, \ldots, \tau - 1\}$, $J = \{\tau, \ldots, e\}$ be two contiguous interval subsets of $\{1, \ldots, n\}$ such that $\min\{|I|, \, |J|\} \geq r + 1$.  Let $\theta = (\theta_i, i = 1, \ldots, n)^{\top} \in {R}^n$ be a piecewise discretized polynomial, i.e.~$\theta_i = f(i/n)$, where $f(\cdot)$ is a polynomial of order at most $r$ on $[s/n, \tau/n)$ and a polynomial of order at most $r$ on $[\tau/n, e/n)$.  

Let $\theta_{I \cup J}$, $\theta$ restricted on $I \cup J$, be reparametrised as 
	\[
		\theta_i = \begin{cases}
 			\sum_{l = 0}^r a_l (i/n - \tau/n)^l, & i \in I, \\
 			\sum_{l = 0}^r b_l (i/n - \tau/n)^l, & i \in J.
 		\end{cases}
	\]
	Then there exists an absolute constant $c_{\mathrm{poly}}$ depending only on $r$ such that for any $d \in \{l = 0, \ldots, r: \, a_l \neq b_l\}$,
	\[
		Q(\theta; I, J) \geq c_{\mathrm{poly}} (a_d - b_d)^2 \frac{\min\{|I|^{2d + 1}, \, |J|^{2d + 1}\}}{n^{2d}}.
	\]
\end{proposition}

\begin{proof}
For any fixed $d \in \{0, \ldots, r\}$ and any $\kappa > 0$, let 
	\begin{align}
		\mathcal{A}_d = \Bigg\{&v \in \mathbb{R}^{|I \cup J|}: \, \mbox{there exist } \{c_{1, l}, c_{2, l}, l = 0, \ldots, r\} \subset \mathbb{R} \nonumber \\
		& \mbox{s.t. } v = \begin{cases}
 			\sum_{l = 0}^r c_{1, l} (i/n - \tau/n)^l, & i \in I, \\
 			\sum_{l = 0}^r c_{2, l} (i/n - \tau/n)^l, & i \in J,
 		\end{cases} \mbox{ and } |c_{1, d} - c_{2, d}| \geq \kappa.\Bigg\} \label{eq-v-a-d-def}
	\end{align}
	In words, $\mathcal{A}_d$ is the set of vectors which are discretised polynomials of order at most $r$ on the interval $I/n$ and different polynomials of order at most $r$ on the interval $J/n$, with the $d$th order coefficients at least $\kappa$ apart. 
	
For $v \in \mathcal{A}_d$, since $v$ is a discretised polynomial on $I/n$ and $J/n$, separately, we have that $Q(v; I, J) = \|v_{I \cup J} - P_{I \cup J} v_{I \cup J}\|^2$.  In addition, we claim that
	\[
		\min_{v \in \mathcal{A}_d} \|v_{I \cup J} - P_{I \cup J} v_{I \cup J}\|^2 = \min_{v \in \mathcal{A}_d} \|v_{I \cup J}\|^2.
	\]
	This is due to the following.  Since orthogonal projections cannot increase the $\ell_2$ norm, we have the $\mathrm{LHS} \leq \mathrm{RHS}$.  As for the other direction, observe that the vector $v_{I \cup J} - P_{I \cup J} v_{I \cup J}$ also belongs to the set $\mathcal{A}_d$. 

It now suffices to lower bound $\min_{v \in \mathcal{A}_d} \|v_{I \cup J}\|^2$.  For any $v \in \mathcal{A}_d$, it holds that
	\begin{align*}
		\|v_{I \cup J}\|^2 =& \|v_I\|^2 + \|v_J\|^2 \geq \frac{c_{\mathrm{poly}}}{n^{2d}} \left(c_{1, d}^2 |I|^{2d + 1}  + c_{2, d}^2 |J|^{2r + 1}\right)\\
		 \geq& \frac{c_{\mathrm{poly}}}{n^{2d}} \kappa^2 \min\{|I|^{2d + 1}, \, |J|^{2d + 1}\},
	\end{align*}
	where $c_{1, d}$ and $c_{2, d}$ are the $d$th order coefficients of $v$ as defined in \eqref{eq-v-a-d-def}, the first inequality is due to \Cref{lem-shen2020}, and the second inequality follows from the fact that $|c_{1, d} - c_{2, d}| \geq \kappa$.
\end{proof}

\begin{lemma}[High Probability Event]\label{lem:noisebd}
Under \Cref{assume-model}, there exists an absolute constant $c_{\mathrm{prob}} > 0$ depending on $r$, and an absolute constant $c_{\mathrm{noise}} > 0$ depending only on $c_{\mathrm{prob}}$, such that
	\[
		\mathbb{P}\left\{\max_{\substack{I = [s, e] \\ 1 \leq s < e \leq n}} \|P_I \varepsilon_I\|^2 \geq c_{\mathrm{noise}} \sigma^2 \log(n)\right\} \leq n^{-c_{\mathrm{prob}}}.
	\]
\end{lemma}

\begin{proof}
For any interval $I \subset \{1, \ldots, n\}$, there exists an absolute positive constant $c > 0$ depending only on $r$ such that for any $t > 0$,
	\[
		\mathbb{P}\left\{\varepsilon_I^{\top} P_I \varepsilon_I - \mathbb{E}\left(\varepsilon_I^{\top} P_I \varepsilon_I\right) \geq t\right\} \leq 2 \exp\left[-c \min\left\{\frac{t^2}{\sigma^4 \|P_I\|_{\mathrm{F}}^2}, \, \frac{t}{\sigma^2 \|P_I\|_{\mathrm{op}}}\right\}\right],
	\]
	which is due to the Hanson--Wright inequality \citep[e.g.~Theorem~1.1 in][]{rudelson2013hanson}.  Since $P_I$ is a rank $r + 1$ orthogonal projection matrix, we have $\|P_I\|_{\mathrm{F}} = r + 1$ and $\|P_I\|_{\mathrm{op}} = 1$.  Then
	\[
		\mathbb{P}\left\{\varepsilon_I^{\top} P_I \varepsilon_I - \mathbb{E}\left(\varepsilon_I^{\top} P_I \varepsilon_I\right) \geq t\right\} \leq 2 \exp\left[-c \min\left\{\frac{t^2}{\sigma^4 (r+1)^2}, \, \frac{t}{\sigma^2}\right\}\right].
	\]
	
In addition, we have that 
	\[
		\mathbb{E}\left(\varepsilon_I^{\top}P_I\varepsilon_I\right) \leq \mathrm{tr}(P_I) \max_{i = 1, \ldots, n} \mathbb{E}(\varepsilon_i^2) \leq (r + 1)\sigma^2.
	\] 
	For an absolute constant $C > c/2$, letting $t = C\sigma^2 \log(n)$ and applying a union bound argument over all possible $I$, we obtain that
	\[
		\mathbb{P}\left\{\max_{\substack{I = [s, e] \\ 1 \leq s < e \leq n}} \|P_I \varepsilon_I\|^2 \geq C\sigma^2 \log(n) + (r+1)\sigma^2 \right\} \leq 2 n^{2-cC}.
	\]
	Finally, we choose $c_{\mathrm{prob}}$ and $c_{\mathrm{noise}}$ such that
	\[
		n^{-c_{\mathrm{prob}}} > 2n^{2-cC} \quad \mbox{and} \quad c_{\mathrm{noise}} \log(n) > C\log(n) + (r+1),
	\]
	then we complete the proof.
\end{proof}

\section{Proof of \Cref{thm}}\label{sec-proofs-main}

In this section, we provide the proof of \Cref{thm}.  We will prove the result by first proving that under an appropriate deterministic choice of the tuning parameter $\lambda$ and some deterministic conditions on other parameters, obtaining the desired localisation error is possible.  We will then conclude the proof using \Cref{lem:noisebd}, under which all these required conditions hold.  

For any $\tau > 0$, define
	\begin{equation}\label{eq-def-G-event}
		\mathcal{M}(\tau) = \left\{\max_{\substack{I = [s, e] \\ 1 \leq s < e \leq n}} \|P_I \varepsilon_I\|^2 \leq \tau \right\}.
	\end{equation}

\begin{proof}[Proof of Theorem~\ref{thm}]
It follows from \Cref{lem:noisebd} that
	\[
		\mathbb{P}\left[\mathcal{M}\{c_{\mathrm{noise}} \sigma^2 \log(n)\}\right] \geq 1 - n^{-c_{\mathrm{prob}}},
	\]
	where $\mathcal{M}(\cdot)$ is defined in \eqref{eq-def-G-event}.  
	
On the event $\mathcal{M}\{c_{\mathrm{noise}} \sigma^2 \log(n)\}$, it follows from \Cref{prop:main1} that 
	\[
		\widehat{K} = K \;\; \mbox{and} \;\; |\widetilde{\eta}_k - \eta_k| \leq c_{\mathrm{error}} \left\{\frac{K n^{2r_k} \sigma^2 \log(n)}{\kappa_k^2}\right\}^{1/(2r_k+1)}, \;\; \forall k \in \{1, \ldots, K\}.
	\]
	In addition, due to \eqref{eq-snr}, it holds that 
	\[
		\max_{k = 1, \ldots, K}|\widetilde{\eta}_k - \eta_k| < \Delta/5.
	\]
	Then it follows from \Cref{lem:refine} that
	\[
		|\widehat{\eta}_k - \eta_k| \leq c_{\mathrm{error}} \left\{\frac{n^{2r_k} \sigma^2 \log(n)}{\kappa_k^2}\right\}^{1/(2r_k+1)}, \quad \forall k \in \{1, \ldots, K\}.
	\]
	We complete the proof.
\end{proof}

\subsection{The initial estimators $\{\widetilde{\eta}_k\}_{k = 1}^{\widehat{K}}$}

The following proposition is our main intermediate result used to prove \Cref{thm}.

\begin{proposition}\label{prop:main1}
Let data $\{y_i\}_{i = 1}^n$ satisfy \Cref{assume-model}.  Let $\{\widetilde{\eta}_k\}_{k = 1}^{\widehat{K}}$ be the initial estimators of \Cref{algorithm:PDP}, with inputs $\{y_i\}_{i = 1}^n$ and tuning parameter $\lambda$.  

On the event $\mathcal{M}(\tau)$ defined in \eqref{eq-def-G-event}, for any $\tau > 0$, let
	\begin{equation}\label{eq-lambda-prop-main-1}
		\lambda > (4 K + 5)\tau.
	\end{equation}
	Assume that
	\begin{equation}\label{eq-rho-prop-main-1}
		\rho > 5^{2r + 1} \max\{6 \lambda, \, r + 1\}.
	\end{equation}
	We have that for any $k \in \{1, \ldots, K\}$, there exists an absolute constant $0< c < 5^{2r/(2r+1)}/2$, such that
	\[
		|\widetilde{\eta}_k - \eta_k| \leq c\left(\frac{n^{2r_k}\max\{6 \lambda, \, r + 1\}}{\kappa_k^2}\right)^{1/(2r_k + 1)}.
	\]
\end{proposition}

\begin{remark}
Note that Proposition~\ref{prop:main1} is a completely deterministic result. In particular, no probabilistic assumption is needed on the noise variables.  The proposition is written with explicit constants but these constants are not optimal in any sense.  We have written out explicit constants just to emphasise the deterministic nature of the result and in better understanding of the relative choices of the different problem parameters. 
\end{remark}

\begin{proof}[Proof of \Cref{prop:main1}]
We will show that 
\begin{itemize}
	\item [(a)]	For any $I = [s, e) \in \widehat{\Pi}$, there are no more than two true change points. 
	\item [(b)]	For any two consecutive intervals $I, J \in \widehat{\Pi}$, the interval $I \cup J$ contains at least one true change point. 
	\item [(c)] For any $I = [s, e) \in \widehat{\Pi}$, if there are exactly two true change points contained in $I$, i.e.~$\eta_{k-1} < s \leq \eta_k < \eta_{k+1} < e \leq \eta_{k+2}$, then 
		\[
			\eta_k - s \leq c\left(\frac{n^{2r_k}\max\{6 \lambda, \, r + 1\}}{\kappa_k^2}\right)^{1/(2r_k + 1)}
		\]
		and
		\[
			e - \eta_{k+1} \leq c\left(\frac{n^{2r_{k+1}}\max\{6 \lambda, \, r + 1\}}{\kappa_{k+1}^2}\right)^{1/(2r_{k+1} + 1)}.
		\]
	\item [(d)] For any $I = [s, e) \in \widehat{\Pi}$, if there is exactly one true change point contained in $I$, i.e.~$\eta_{k-1} < s \leq \eta_k < e \leq \eta_{k+1}$, then
		\[
			\min\{e - \eta_k, \, \eta_k - s\} \leq c\left(\frac{n^{2r_k}\max\{6 \lambda, \, r + 1\}}{\kappa_k^2}\right)^{1/(2r_k + 1)}.
		\]
	\item [(e)] If $|\Pi| \leq |\widehat{\Pi}| \leq 3|\Pi|$, then it holds that $|\widehat{\Pi}| = |\Pi|$. 
\end{itemize}

Parts (a)-(d) are shown in Lemmas~\ref{lem-no-more-than-two}, \ref{lem-partd}, \ref{lem-exactly-two}, \ref{lem-case-c} and \ref{sec-proof-e}, respectively.  Letting $\widetilde{\eta}_0 = 1$ and $\widetilde{\eta}_{\widehat{K} + 1} = n+1$, it follows from part (b) that for every 3 consecutive change point estimators in $\{\widetilde{\eta}_k\}_{k = 0}^{\widehat{K} + 1}$, there is at least one true change point $\{\eta_k\}_{k = 0}^{K}$.  This implies that $|\widehat{\Pi}| \leq 3 |\Pi|$. 

In addition, by part (a), an interval $I = [s, e) \in \widehat{\Pi}$ can contain two, one or zero true change point.  If $I$ contains exactly two true change points, then by part (c), the smaller true change point is close to the left endpoint $s$, and the larger true change point is close to the right endpoint $e$.  The closeness is defined by part (c).  If $I$ contains exactly one true change point, then by part (d), the true change point is close to one of the endpoints.  This shows that every true change point can be mapped to an estimated change point, and the distance between the true and the estimated is upper bounded by what is shown in (c) and (d). 

Recall the definition of $\rho$ in \Cref{assume-model} and the condition \eqref{eq-rho-prop-main-1}.  We have that
	\begin{align*}
		\Delta & > \max_{k = 1, \ldots, K} \left(\frac{5^{2r+1} \max\{6\lambda, r+1\} n^{2r_k}}{\kappa_k^2}\right)^{1/(2r_k+1)} \\
		& > 2c \max_{k = 1, \ldots, K} \left(\frac{ \max\{6\lambda, r+1\} n^{2r_k}}{\kappa_k^2}\right)^{1/(2r_k+1)}.
	\end{align*}
	This assures that the mapping of true change points to estimated change points is one to one and implies that $|\widehat{\Pi}| \geq |\Pi|$.  Finally, part (e) is deployed to complete the proof. 
\end{proof}

\begin{lemma}[Part (a) in the proof of \Cref{prop:main1}]\label{lem-no-more-than-two}
Under all the assumptions in \Cref{prop:main1}, for any $I \in \widehat{\Pi}$, it holds that $I$ does not contain more than two true change points.
\end{lemma}

\begin{proof}
We prove by contradiction, assuming that there exists at least three true change points in $I = [s, e) \in \widehat{\Pi}$, namely $s \leq \eta_{k - 1} < \eta_k < \eta_{k + 1} < e$.  This implies that
	\[
		\min\{\eta_k - s, \, e - \eta_k\} > \Delta.
	\]
	Denote $I_1 = [s, \eta_k - \Delta)$, $I_2 = [\eta_k - \Delta, \eta_k)$, $I_3 = [\eta_k, \eta_k + \Delta)$ and $I_4 = [\eta_k + \Delta, e)$.  Let $\widetilde{\Pi}$ be the interval partition such that 
	\[
		\widetilde{\Pi} = \widehat{\Pi} \cup \{I_1, I_2, I_3, I_4\} \setminus \{I\}.
	\]	

It holds that	
	\begin{align*}
		0 & \geq G(\widehat{\Pi}, \lambda) - 	G(\widetilde{\Pi}, \lambda) \\
		& = -3\lambda + H(y, I) - H(y, I_1) - H(y, I_2) - H(y, I_3) - H(y, I_4) \\
		& \geq -3 \lambda + H(y, I_2 \cup I_3) - H(y, I_2) - H(y, I_3) = -3 \lambda + Q(y; I_2,I_3) \\
		& \geq -3\lambda + \left\{\sqrt{Q(\theta; I_2, I_3)} - \sqrt{Q(\varepsilon; I_2,I_3)}\right\}^2 \\
		& \geq -3\lambda + \frac{Q(\theta; I_2, I_3)}{4} \mathbbm{1}\{Q(\theta; I_2, I_3) > 4 Q(\varepsilon; I_2, I_3)\},
	\end{align*}
	where $\mathbbm{1}\{\cdot\}$ is an indicator function, the first inequality follows the definition of $\widehat{\Pi}$, the second is from \Cref{lem:1} and the third follows from \Cref{lem:2}.  As for the final inequality, it follows from \Cref{prop:1} and the fact that $|I_2|, |I_3| = \Delta$ that $Q(\theta; I_2, I_3) \geq \rho$.  Since our assumption implies that $2 \tau \leq \rho$, it holds that $12 \lambda \geq \rho$ which contradicts the second assumption in~\eqref{eq-lambda-prop-main-1}. 
\end{proof}

\begin{lemma}[Part (b) in the proof of \Cref{prop:main1}] \label{lem-partd}
Under all the assumptions in \Cref{prop:main1}, for any two consecutive intervals $I_1, I_2  \in \widehat{\Pi}$, there is at least one true change point in $I_1 \cup I_2$. 
\end{lemma}

\begin{proof}
We prove by contradiction, assuming there is no true change point in $J = I_1 \cup I_2$.  Let
	\[
		\widetilde{\Pi} = \widehat{\Pi} \cup \{J\} \setminus \{I_1, I_2\}.
	\]
	We have that
	\[
		0 \leq G(\widetilde{\Pi}, \lambda) - G(\widehat{\Pi}, \lambda) = -\lambda + Q(y; I_1,I_2) = -\lambda + Q(\varepsilon; I_1, I_2) \leq -\lambda + 3\tau,
	\]
	where the first inequality is due to the definition of $\widehat{\Pi}$, the second identity follows from the fact that $\theta$ is polynomial of order at most $r$ on $J$, and the last inequality holds on the event $\mathcal{M}$.  Therefore we reach a contradiction to \eqref{eq-lambda-prop-main-1}.
\end{proof}

\begin{lemma}[Part (c) in the proof of \Cref{prop:main1}]\label{lem-exactly-two}
Under all the assumptions in \Cref{prop:main1}, for any $I = [s, e) \in \widehat{\Pi}$, if there are exactly two true change points $\eta_k, \eta_{k+1} \in I$, then it holds that
	\[
		\eta_k - s \leq c\left(\frac{n^{2r_k}\max\{2\lambda + 12\tau, \, r + 1\}}{\kappa_k^2}\right)^{1/(2r_k + 1)}
	\]
	and
	\[
		e - \eta_{k+1} \leq c\left(\frac{n^{2r_{k+1}}\max\{2\lambda + 12\tau, \, r + 1\}}{\kappa_{k+1}^2}\right)^{1/(2r_{k+1} + 1)}.
	\]
\end{lemma}

\begin{proof}
Let $I_1 = [s, \eta_k)$, $I_2 = [\eta_k, \eta_{k+1})$, $I_3 = [\eta_{k+1}, e)$ and
	\[
		\widetilde{\Pi} = \widehat{\Pi} \cup \{I_1, I_2 \cup I_3\} \setminus \{I\}.
	\]
	It holds that
	\begin{align*}
		0 & \leq G(\widetilde{\Pi}, \lambda) - G(\widehat{\Pi}, \lambda) \leq \lambda - H(y, I) + H(y, I_1) + H(y, I_2 \cup I_3) \\
		& \leq \lambda - H(y, I_2) - H(y, I_3) + H(y, I_2 \cup I_3) = \lambda - Q(y; I_2, I_3)  \\
		& \leq \lambda - |\sqrt{Q(\theta; I_2, I_3)} - \sqrt{Q(\varepsilon; I_2, I_3)}|^2 \leq \lambda - Q(\theta; I_2, I_3)/2 + 2 Q(\varepsilon; I_2, I_3) \\
		& \leq \lambda + 6\tau - Q(\theta; I_2, I_3)/2,
	\end{align*}
	where the first inequality follows from the definition of $\widehat{\Pi}$, the third inequality follows from \Cref{lem:1}, the fourth inequality follows from \Cref{lem:2}, the fifth inequality follows from $(a-b)^2 > a^2/ - 2b^2$, for any $a, b \in \mathbb{R}$, and the last follows from the definition of the event $\mathcal{M}(\tau)$.
	
Applying \Cref{prop:1}, we conclude that 
	\[
		c_{\mathrm{poly}} \frac{\kappa_{k + 1}^2}{n^{2r_{k + 1}}} \min\{\Delta^{2r_{k + 1} + 1}, \, |I_3|^{2r_{k + 1} + 1}\}\leq \max\{12\tau + 2\lambda, \, r + 1\}.
	\]
	It follows from \eqref{eq-rho-prop-main-1}, we have that
	\[
		e - \eta_{k+1} \leq c\left(\frac{n^{2r_{k+1}}\max\{2\lambda + 12\tau, \, r + 1\}}{\kappa_{k+1}^2}\right)^{1/(2r_{k+1} + 1)}.
	\]
	The same arguments can lead to the corresponding result on $\eta_k - s$ and complete the proof.
\end{proof}

\begin{lemma}[Part (d) in the proof of \Cref{prop:main1}]\label{lem-case-c}
Under all the assumptions in \Cref{prop:main1}, for any $I = [s, e) \in \widehat{\Pi}$, if there is exactly one true change point $\eta_k \in I$, then
	\[
		\min\{e - \eta_k, \, \eta_k - s\} \leq c\left(\frac{n^{2r_k}\max\{12\tau + 2\lambda, \, r + 1\}}{\kappa_k^2}\right)^{1/(2r_k + 1)}.
	\]
\end{lemma}

\begin{proof}
Let $I_1 = [s, \eta_k)$, $I_2 = [\eta_k, e)$ and
	\[
		\widetilde{\Pi} = \widehat{\Pi} \cup \{I_1, I_2\} \setminus \{I\}.
	\]
	It holds that
	\begin{align*}
		0 & \leq G(\widetilde{\Pi}, \lambda) - G(\widehat{\Pi}, \lambda) = \lambda - H(y, I) + H(y, I_1) + H(y, I_2) = \lambda - Q(y;I_1,I_2) \\
		& \leq \lambda - |\sqrt{Q(\theta; I_1, I_2)} - \sqrt{Q(\varepsilon; I_1, I_2)}|^2 \leq \lambda - Q(\theta; I_2, I_3)/2 + 2 Q(\varepsilon; I_2, I_3) \\
		& \leq \lambda + 6\tau - Q(\theta; I_2, I_3)/2,
	\end{align*}
	where the first inequality follows from the definition of $\widehat{\Pi}$, the second inequality follows from \Cref{lem:2}, the third inequality follows from $(a-b)^2 > a^2/ - 2b^2$, for any $a, b \in \mathbb{R}$, and the last follows from the definition of the event $\mathcal{M}(\tau)$.

Applying \Cref{prop:1}, we conclude that 
	\[
		c_{\mathrm{poly}} \frac{\kappa_k^2}{n^{2r_k}} \min\{|I_1|^{2r_k + 1}, \, |I_2|^{2r_k + 1}\}\leq \max\{12\tau + 2\lambda, \, r + 1\}.
	\]
\end{proof}

\begin{lemma}[Part (e) in the proof of \Cref{prop:main1}]\label{sec-proof-e}
Under all the assumptions in \Cref{prop:main1}, assuming that $|\Pi| \leq |\widehat{\Pi}| \leq 3 |\Pi|$, it holds that $|\widehat{\Pi}| = |\Pi|$.
\end{lemma}

\begin{proof}
To ease notation, for any interval partition $\mathcal{P}$ of $\{1, \ldots, n\}$ and any $v \in \mathbb{R}^n$, we let
	\[
		S(v, \mathcal{P}) = \sum_{I \in \mathcal{P}} H(v, I).
	\]
	Using this notation, we first note that
	\[
		\|\varepsilon\|^2 \geq S(y, \Pi),
	\] 
	since for any $I \in \Pi$, $H(y, I) = H(\varepsilon, I) \leq \|\varepsilon_I\|^2$.  

Let $\widehat{\Pi} \cap \Pi$	 be the intersection of the partitions $\widehat{\Pi}$ and $\Pi$.  It then holds that
	\begin{align}
		\|\varepsilon\|^2 + (K + 1) \lambda \geq & S(y, \Pi) + (K + 1) \lambda \geq S(y, \widehat{\Pi}) + (\widehat{K} + 1) \lambda \nonumber \\ 
		\geq & S(y, \widehat{\Pi} \cap \Pi) + (\widehat{K} + 1) \lambda, \label{eq-prf-e-1}
	\end{align}
	where the second inequality follows from the definition of $\widehat{\Pi}$ and the last inequality is due to \Cref{lem:1}. 
	
On the other hand, we have that
	\begin{equation}\label{eq-prf-e-2}
		\|\varepsilon\|^2 - S(y, \widehat{\Pi} \cap \Pi) = \|\varepsilon\|^2 - S(\varepsilon, \widehat{\Pi} \cap \Pi) \leq \big(\widehat{K} + K + 2\big) \tau,
	\end{equation}
	where the identity is due to the fact that $\theta$ is a polynomial of order at most $r$ on every member of $\widehat{\Pi} \cap \Pi$, and the second inequality holds on the event $\mathcal{M}(\tau)$, noticing that $|\widehat{\Pi} \cap \Pi| \leq \widehat{K} + K + 2$.
	
Combining \eqref{eq-prf-e-1} and \eqref{eq-prf-e-2}, we have that		
	\[
		\big(\widehat{K} - K\big) \lambda \leq \big(\widehat{K} + K + 2\big) \tau \leq \big(4K + 5\big) \tau,
	\]
	where the last inequality is due to $|\widehat{\Pi}| \leq 3|\Pi|$.  Since we also have $|\widehat{\Pi}| \geq |\Pi|$, the last display implies that $|\widehat{K}| = K$, otherwise it contradicts with \eqref{eq-lambda-prop-main-1}.
\end{proof}

\subsection{The final estimators $\{\widehat{\eta}_k\}_{k = 1}^{\widehat{K}}$}

The following lemma shows that our update step in \Cref{algorithm:PDP} can significantly improve the initial estimators.

\begin{lemma}\label{lem:refine}
Let data $\{y_i\}_{i = 1}^n$ satisfy \Cref{assume-model}.  For any set $\{\nu_k\}_{k = 1}^K$ satisfying
	\[
		\max_{k = 1, \ldots, K}|\nu_k - \eta_k| \leq \Delta/5,
	\]
	with $\nu_0 = 1$ and $\nu_{K+1} = n+1$, define
	\[
		s_k = \nu_{k-1}/2 + \nu_k/2, \;\; e_k = \nu_k/2 + \nu_{k+1}/2 \;\; \mbox{and} \;\; I_k = [s_k, e_k), \;\; \forall k \in \{1, \ldots, K\}.
	\]
	Let
	\[
		\widehat{\eta}_k = \argmin_{t \in I_k \setminus \{s_k\}} \left\{H(y, [s_k, t)) + H(y, [t, e_k))\right\}, \quad \forall k \in \{1, \ldots, K\}.
	\]
	
For any $\tau > 0$, if 
	\begin{equation}\label{eq-rho-lemb7}
		\rho > 2\times 5^{2r + 1} \tau, 		
	\end{equation}
	then on the event $\mathcal{M}(\tau)$, it holds that for an absolute constant $c > 0$,
	\[
		|\widehat{\eta}_k - \eta_k| \leq c \left\{\frac{n^{2r_k} \tau}{\kappa_k^2}\right\}^{1/(2r_k+1)}.
	\]
\end{lemma}

\begin{proof}
For any $k \in \{1, \ldots, K\}$, we have that
	\begin{align*}
		\eta_k - s_k =& \eta_k - \nu_k + \frac{\nu_k - \nu_{k - 1}}{2}\\
		=& \eta_k - \nu_k +  \frac{\nu_k - \eta_{k}}{2} +  \frac{\eta_k - \eta_{k - 1}}{2} +  \frac{\eta_{k - 1} - \nu_{k - 1}}{2} \\
		\geq & -\frac{\Delta}{5} - \frac{\Delta}{10} + \frac{\Delta}{2} - \frac{\Delta}{10} = \frac{\Delta}{10}.
	\end{align*}
	and
	\begin{align*}
		s_k - \eta_{k-1} =&  \frac{\nu_k - \nu_{k - 1}}{2} + \nu_{k - 1} - \eta_{k - 1}\\
		=& \frac{\nu_k - \eta_{k}}{2} + \frac{\eta_k - \eta_{k - 1}}{2} + \frac{\eta_{k - 1} - \nu_{k - 1}}{2} + \nu_{k - 1} - \eta_{k - 1} \\
		\geq & -\frac{\Delta}{10} + \frac{\Delta}{2} - \frac{\Delta}{10} - \frac{\Delta}{5} = \frac{\Delta}{10}.
	\end{align*}
	Using identical arguments, it also holds that $\min\{e_k - \eta_k, \eta_{k + 1} - e_k\} \geq \Delta/10$.  
	
Without loss of generality, we assume that
	\[
		s_k < \widehat{\eta}_k < \eta_k < e_k.
	\]
	Let $J_1 = [s_k, \widetilde{\eta}_k)$, $J_2 = [\widetilde{\eta}_k, \eta_k)$ and $J_3 = [\eta_k, e_k)$.  By the definition of $\widetilde{\eta}_k$, we have that
	\begin{align*}
		& H(y, J_1 \cup J_2) + H(y, J_3) \geq H(y, J_1) + H(y, J_2 \cup J_3) \\
		= & H(y, J_1) + H(y, J_2) + H(y, J_3) + Q(y; J_2, J_3).
	\end{align*} 
	We then have that
	\begin{align*}
	c_{poly} \frac{\kappa_{k}^2}{n^{2r_k}} \min\{|J_2|^{2r_k + 1}, |J_3|^{2r_k + 1}\} \leq& Q(y;J_2,J_3) \leq Q(y; J_1, J_2) = Q(\varepsilon; J_1, J_2)\\
	\leq& 3 \tau,
	\end{align*}
	where the first inequality is due to Proposition~\ref{prop:1}.  Since $|J_3| \geq \frac{\Delta}{10}$, the final claim follows from \eqref{eq-rho-lemb7}.
\end{proof}

\section{Proofs of Lemmas~\ref{lem-lb-2}, \ref{lem-lb-1} and \ref{lem-lb-3333}}\label{sec-proofs-lb}

\begin{proof}[Proof of \Cref{lem-lb-2}]
Let $P_0$ denote the joint distribution of the independent random variables $\{y_i\}_{i = 1}^n$ such that
	\[
		y_i \sim \begin{cases}
 			\mathcal{N}(0, \sigma^2), & i \in \{1, \ldots, \Delta\} \\
 			\mathcal{N}(\kappa (i/n - \Delta/n)^r, \sigma^2), & i \in \{\Delta + 1, \ldots, n\}. 
 		\end{cases}
	\]
	Let $P_1$ denote the joint distribution of the independent random variables $\{z_i\}_{i = 1}^n$ such that
	\[
		z_i \sim \begin{cases}
 			\mathcal{N}(0, \sigma^2), & i \in \{1, \ldots, \Delta + \delta\} \\
 			\mathcal{N}(\kappa (i/n - \Delta/n)^r, \sigma^2), & i \in \{\Delta + \delta +  1, \ldots, n\},
 		\end{cases}
	\]
	where $\delta$ is a positive integer no larger than $\Delta$.  We further assume that $\sigma^2 \log(n) \kappa^{-2} \geq 1$ and $\Delta \leq n/4$.
	
As for $P_0$, it is easy to see that 
	\[
		\mathbb{E}(y_i) = \begin{cases}
 			0, & i \in \{1, \ldots, \Delta\}, \\
 			\kappa \left\{(i-\Delta)/n\right\}^r, & i \in \{\Delta +  1, \ldots, n\},
 		\end{cases}
	\]
	which implies that the change point of $P_0$ satisfies $\eta(P_0) = \Delta + 1$.  Recalling \eqref{eq-active-rk-kk}, we also know that the corresponding order $r_1$ equals $r$, and the jump size $\kappa_1 = \kappa$. 
	
As for $P_1$, we have that
	\[
		\mathbb{E}(y_i) = \begin{cases}
 			0, & i \in \{1, \ldots, \Delta + \delta\}, \\
 			\kappa \left\{(i-\Delta)/n\right\}^r = \kappa \sum_{l = 0}^r {r \choose l} \left(\frac{i-\Delta-\delta}{n}\right)^{r-l} \left(\frac{\delta}{n}\right)^l, & i \in \{\Delta + \delta + 1, \ldots, n\},
 		\end{cases}
	\]
	which implies that the change point of $P_1$ satisfies $\eta(P_1) = \Delta + \delta + 1$.  Recalling \Cref{def-jump-size} and \eqref{eq-active-rk-kk}, we have that for $l \in \{0, \ldots, r\}$, 
	\[
		\rho_{1, l} = \kappa^2\left(\frac{\delta}{n}\right)^{2r-2l} (\Delta + \delta)^{2l+1} n^{-2l}
	\]
	and
	\begin{equation}\label{eq-pf-lb-333-1}
		\left(\frac{\sigma^2 \log(n)}{\rho_{1, l}}\right)^{1/(2l+1)} = \left(\frac{\sigma^2 \log(n)}{\kappa^2}\right)^{1/(2l+1)} \frac{1}{\Delta + \delta} \delta^{\frac{2l-2r}{2l+1}} n^{\frac{2r}{2l+1}}.
	\end{equation}
	In the above, we have used the condition that $\delta \leq \Delta < n/4$.  Due to the assumption that $\sigma^2 \log(n) \kappa^{-2} \geq 1$, we have that \eqref{eq-pf-lb-333-1} is a decreasing function of the order $r$.  Therefore by the definition in \eqref{eq-active-rk-kk}, we have that the corresponding order $r_1 = r$ and the jump size $\kappa_1 = \kappa$. 
	
It then follows from Le Cam's lemma \citep[e.g.][]{yu1997assouad}, a standard reduction of estimation to two point testing, and Lemma~2.6 in \cite{Tsybakov2009}, a form of Pinsker's inequality, that
	\begin{align*}
		& \inf_{\widehat{\eta}} \sup_{P \in \mathcal{Q}} \mathbb{E}_P(|\widehat{\eta} - \eta|) \geq \delta (1 - d_{\mathrm{TV}}(P_0, P_1)) \geq \frac{\delta}{2} \exp\left( -\frac{\kappa^2}{\sigma^2 n^{2r}} \sum_{i = \Delta + 1}^{\Delta + \delta}(i - \Delta)^{2r} \right) \\
		= & \frac{\delta}{2} \exp\left( -\frac{\kappa^2}{\sigma^2 n^{2r}} \sum_{i = 1}^{\delta}i^{2r} \right) \geq \frac{\delta}{2} \exp\left( -\frac{\kappa^2}{\sigma^2 n^{2r}} \int_1^{\delta + 1} x^{2r}\,dx \right)\\
		\geq& \frac{\delta}{2} \exp\left\{ -\frac{\kappa^2 (\delta+1)^{2r+1}}{(2r+1)\sigma^2n^{2r}}  \right\}
		\geq \frac{\delta}{2} \exp\left\{-\frac{c\kappa^2\delta^{2r+1}}{\sigma^2 n^{2r}} \right\}.
	\end{align*}
	We set 
	\[
		\delta = \max\left\{\left[\frac{c \sigma^2 n^{2r}}{\kappa^2}\right]^{1/(2r+1)},\, 1\right\}
	\]
	and complete the proof.
\end{proof}

\begin{proof}[Proof of \Cref{lem-lb-1}]
Let $P_0$ denote the joint distribution of the independent random variables $\{y_i\}_{i = 1}^n$ such that
	\[
		y_i \sim \begin{cases}
 			\mathcal{N}(\kappa (i/n - \Delta/n)^r, \sigma^2), & i \in \{1, \ldots, \Delta\} \\
 			\mathcal{N}(0, \sigma^2), & i \in \{\Delta + 1, \ldots, n\}. 
 		\end{cases}
	\]
	Let $P_1$ denote the joint distribution of the independent random variables $\{z_i\}_{i = 1}^n$ such that
	\[
		z_i \sim \begin{cases}
 			\mathcal{N}(0, \sigma^2), & i \in \{1, \ldots, n-\Delta\} \\
 			\mathcal{N}(\kappa (i/n - 1 + \Delta/n)^r, \sigma^2), & i \in \{n- \Delta + 1, \ldots, n\}. 
 		\end{cases}
	\]
	
As for $P_0$, it is easy to see that 
	\[
		\mathbb{E}(y_i) = \begin{cases}
 			\kappa \left\{(i-\Delta)/n\right\}^r, & i \in \{1, \ldots, \Delta\}, \\
 			0, & i \in \{\Delta + 1, \ldots, n\},
 		\end{cases}
	\]
	which implies that the change point of $P_0$ satisfies $\eta(P_0) = \Delta + 1$. Recalling \eqref{eq-active-rk-kk}, we also know that the corresponding order $r_1$ equals $r$, and the jump size $\kappa_1 = \kappa$. 

As for $P_1$, it is easy to see that 
	\[
		\mathbb{E}(y_i) = \begin{cases}
 			0, & i \in \{1, \ldots, n-\Delta\}, \\
 			\kappa \left\{(i-n+\Delta)/n\right\}^r, & i \in \{n-\Delta + 1, \ldots, n\},
 		\end{cases}
	\]
	which implies that the change point of $P_1$ satisfies $\eta(P_1) = n-\Delta + 1$.  Recalling \Cref{def-jump-size}, we also know that the corresponding smallest order $r$ equals $r$, and the jump size $\kappa = \kappa$. 
	
Since $\Delta \leq n/3$, it follows from Le Cam's lemma \citep[e.g.][]{yu1997assouad} and Lemma~2.6 in \cite{Tsybakov2009} that
	\[
		\inf_{\widehat{\eta}} \sup_{P \in \mathcal{P}} \mathbb{E}_P(|\widehat{\eta} - \eta|) \geq (n/3) (1 - d_{\mathrm{TV}} (P_0, P_1)) \geq \frac{n}{6} \exp\{-\mathrm{KL}(P_0, P_1)\}.
	\]
	Since both $P_0$ and $P_1$ are product measures, it holds that
	\begin{align*}
		\mathrm{KL}(P_0, P_1) & = \frac{\kappa^2}{\sigma^2} \left\{\sum_{i = 1}^{\Delta} \left(\frac{i-\Delta}{n}\right)^{2r} + \sum_{i = n-\Delta + 1}^n \left(\frac{i-n+\Delta}{n}\right)^{2r}\right\} \\
		& \leq \frac{2\kappa^2}{\sigma^2}\sum_{i = 1}^{\Delta} \left(\frac{i}{n}\right)^{2r} \leq \frac{2\kappa^2}{\sigma^2 n^{2r}} \int_1^{\Delta+1} x^{2r}\, dx \leq \frac{c_0\kappa^2 \Delta^{2r+1}}{\sigma^2 n^{2r}} = c_0\xi.
	\end{align*}
	Therefore
	\[
		\inf_{\hat{\eta}} \sup_{P \in \mathcal{P}} \mathbb{E}_P(|\hat{\eta} - \eta|) \geq \frac{n}{6\exp(c_0 \xi)} = cn.
	\]

\end{proof}

\begin{proof}[Proof of \Cref{lem-lb-3333}]
Let $P_0$ denote the joint distribution of the independent random variables $\{y_i\}_{i = 1}^n$ such that
	\[
		y_i \sim \begin{cases}
 			\mathcal{N}(0, \sigma^2), & i \in \{1, \ldots, \Delta\} \\
 			\mathcal{N}(\kappa (i/n - \Delta/n)^r, \sigma^2), & i \in \{\Delta + 1, \ldots, n\}. 
 		\end{cases}
	\]
	Let $P_1$ denote the joint distribution of the independent random variables $\{z_i\}_{i = 1}^n$ such that
	\[
		z_i \sim \begin{cases}
 			\mathcal{N}(0, \sigma^2), & i \in \{1, \ldots, \Delta + \delta\} \\
 			\mathcal{N}(\kappa (i/n - \Delta/n - \delta/n)^r, \sigma^2), & i \in \{\Delta + \delta +  1, \ldots, n\},
 		\end{cases}
	\]
	where $\delta$ is a positive integer no larger than $\Delta$.  We further assume that $\Delta \leq n/4$.
	
As for $P_0$, it is easy to see that 
	\[
		\mathbb{E}(y_i) = \begin{cases}
 			0, & i \in \{1, \ldots, \Delta\}, \\
 			\kappa \left\{(i-\Delta)/n\right\}^r, & i \in \{\Delta +  1, \ldots, n\},
 		\end{cases}
	\]
	which implies that the change point of $P_0$ satisfies $\eta(P_0) = \Delta + 1$.  Recalling \eqref{eq-active-rk-kk}, we also know that the corresponding order $r_1$ equals $r$, and the jump size $\kappa_1 = \kappa$.  In addition, at the change point, under the reparametrisation, $a_{1, l} = b_{1, l}$, $l \in \{0, \ldots, r - 1\}$.
	
As for $P_1$, we have that
	\[
		\mathbb{E}(y_i) = \begin{cases}
 			0, & i \in \{1, \ldots, \Delta + \delta\}, \\
 			\kappa \left\{(i-\Delta - \delta)/n\right\}^r, & i \in \{\Delta + \delta + 1, \ldots, n\},
 		\end{cases}
	\]
	which implies that the change point of $P_1$ satisfies $\eta(P_1) = \Delta + \delta + 1$.  Recalling \eqref{eq-active-rk-kk}, we also know that the corresponding order $r_1$ equals $r$, and the jump size $\kappa_1 = \kappa$.  In addition, at the change point, under the reparametrisation, $a_{1, l} = b_{1, l}$, $l \in \{0, \ldots, r - 1\}$.

It then follows from Le Cam's lemma \citep[e.g.][]{yu1997assouad}, a standard reduction of estimation to two point testing, and Lemma~2.6 in \cite{Tsybakov2009}, a form of Pinsker's inequality, that
	\begin{align}
		& \inf_{\widehat{\eta}} \sup_{P \in \mathcal{Q}} \mathbb{E}_P(|\widehat{\eta} - \eta|) \geq \delta (1 - d_{\mathrm{TV}}(P_0, P_1)) \nonumber \\
		\geq & \frac{\delta}{2} \exp\left( -\frac{\kappa^2}{\sigma^2 n^{2r}} \sum_{i = \Delta + 1}^{\Delta + \delta}(i - \Delta)^{2r} \right)\\
		& \times  \exp\left( -\frac{\kappa^2}{\sigma^2 n^{2r}} \sum_{i = \Delta + \delta + 1}^n \left\{(i - \Delta)^r - (i - \Delta - \delta)^r \right\}^2\right) \nonumber \\
		= & \frac{\delta}{2} \mbox{(I)} \times \mbox{(II)}.\label{eq-aaaaaaaa}
	\end{align}
	
As for the term (I), we have that 
	\begin{align*}		
		\mbox{(I)} = & \exp\left( -\frac{\kappa^2}{\sigma^2 n^{2r}} \sum_{i = 1}^{\delta}i^{2r} \right) \geq  \exp\left( -\frac{\kappa^2}{\sigma^2 n^{2r}} \int_1^{\delta + 1} x^{2r}\,dx \right)\\
		\geq&  \exp\left\{ -\frac{\kappa^2 (\delta+1)^{2r+1}}{(2r+1)\sigma^2n^{2r}}  \right\}
		\geq \exp\left\{-\frac{c\kappa^2\delta^{2r+1}}{\sigma^2 n^{2r}} \right\}.
	\end{align*}
	In order to ensure that $\mbox{(I)} \gtrsim 1$, we need
	\begin{equation}\label{eq-bbbbbbbb}
		\delta \lesssim \left(\frac{\sigma^2 n^{2r}}{\kappa^2}\right)^{1/(2r+1)}.
	\end{equation}

As for the term (II), we have that 	
	\begin{align*}
		\mbox{(II)} = &	 \exp\left( -\frac{\kappa^2}{\sigma^2 n^{2r}} \sum_{i = 1}^{n-\Delta - \delta} \left\{\sum_{l = 0}^{r-1}{r \choose l} i^l \delta^{r-l} \right\}^2\right)\\
		\geq& \exp\left( -\frac{r\kappa^2}{\sigma^2 n^{2r}} \sum_{i = 1}^{n-\Delta - \delta} \sum_{l = 0}^{r-1}{r \choose l}^2 i^{2l} \delta^{2r-2l} \right) \\
		\geq & \exp\left( -\frac{r\kappa^2}{\sigma^2 n^{2r}} \sum_{l = 0}^{r-1}  {r \choose l}^2 \delta^{2r-2l}\int_{i = 1}^{n-\Delta - \delta +1}  x^{2l} \, dx \right) \\
		\geq & \exp\left( -\frac{r\kappa^2}{\sigma^2 n^{2r}} \sum_{l = 0}^{r-1}  {r \choose l}^2 \delta^{2r-2l} \frac{(n-\Delta - \delta +1)^{2l+1}}{2l+1} \right) \\
		\geq & \exp\left( -\frac{r\kappa^2}{\sigma^2 n^{2r}} \sum_{l = 0}^{r-1}  {r \choose l}^2  \frac{\delta^{2r-2l} n^{2l+1}}{2l+1} \right)\\
		=&  \exp\left( -\frac{r\kappa^2}{\sigma^2 } \sum_{l = 0}^{r-1}  {r \choose l}^2  \frac{n}{2l+1}  \left(\frac{\delta}{n}\right)^{2r-2l} \right).
	\end{align*}
	Since $\delta < n$, we have that with a constant $C$ being a function of $r$, it holds that
	\begin{align*}
		\mbox{(II)} \geq \exp\left( -\frac{C\kappa^2}{\sigma^2 } n\left(\frac{\delta}{n}\right)^2 \right).
	\end{align*}
	In order to ensure that $\mbox{(II)} \gtrsim 1$, we need
	\begin{equation}\label{eq-ccccccccc}
		\delta \lesssim \left(\frac{\sigma^2 n}{\kappa^2}\right)^{1/2}.
	\end{equation}

Combining \eqref{eq-aaaaaaaa}, \eqref{eq-bbbbbbbb} and \eqref{eq-ccccccccc}, setting 
	\[
		\delta = \max\left\{\left[\frac{c \sigma^2 n}{\kappa^2}\right]^{1/2},\, 1\right\},
	\]
	we complete the proof.
\end{proof}




\bibliographystyle{plainnat}
\bibliography{citation} 


\end{document}